\documentclass[]{interact}

\usepackage{fix-cm}
\usepackage[utf8]{inputenc}
\usepackage{graphicx}
\usepackage{geometry}
\usepackage{tabto}
\usepackage{subfigure}
\usepackage{enumitem}
\usepackage{xfrac}
\usepackage{paracol}
\usepackage{pifont}
\usepackage{mathtools}
\usepackage{tikz}
\usepackage{calc}
\usepackage{parcolumns}
\usepackage{diagbox}
\usepackage{hyperref}
\usepackage[all]{nowidow}
\usepackage[natbibapa,nodoi]{apacite}
\renewcommand{\t}{1}
\newcommand{\f}{0}
\renewcommand{\i}{\sfrac{1}{2}}

\newcommand{\X}{\mathbb{X}}
\newcommand{\Y}{\mathbb{Y}}
\newcommand{\Z}{\mathbb{Z}}

\newcommand{\CL}{\mathbb{CL}}

\DeclareUnicodeCharacter{0080}{ }

\UseRawInputEncoding

\theoremstyle{plain}
\newtheorem{theorem}{Theorem}[section]
\newtheorem{proposition}[theorem]{Proposition}
\newtheorem{lemma}[theorem]{Lemma}

\theoremstyle{definition}
\newtheorem{Definition}[theorem]{Definition}
\newtheorem{fact}[theorem]{Fact}

\newtheorem{corollary}[theorem]{Corollary}

\begin{document}
\sloppy

\title{Reaching Classicality through Transitive Closure}

\author{
\name{Quentin Blomet\textsuperscript{a,b} and Bruno Da R{\'e}\textsuperscript{c,d}}
\affil{\textsuperscript{a}Institut Jean-Nicod (CNRS, ENS-PSL, EHESS), PSL University, Paris, France; 
\textsuperscript{b}Department of Philosophy, University of Greifswald, Germany;\\
\textsuperscript{c}IIF (CONICET-SADAF), Buenos Aires, Argentina;\\
\textsuperscript{d}Department of Philosophy, University of Buenos Aires, Argentina}}

\maketitle

\begin{abstract}
Recently, \cite{da2023three} showed that all 
logics based on \textit{Boolean Normal monotonic three-valued schemes} coincide with classical logic when defined using a \textit{strict-tolerant} standard ($\mathbf{st}$). Conversely, they proved that under a \textit{tolerant-strict} standard ($\mathbf{ts}$), the resulting logics are all empty. 
Building on these results, we show that classical logic can be obtained by closing under transitivity the union of two logics defined over (potentially different) Boolean normal monotonic schemes, using a strict-strict standard 
($\mathbf{ss}$) for one and a tolerant-tolerant standard ($\mathbf{tt}$) for the other, with the first of these logics being paracomplete and the other being paraconsistent. We then identify a notion dual to transitivity that allows us to characterize the logic $\mathsf{TS}$ as the dual transitive closure of the intersection of any two logics defined over (potentially different) Boolean normal monotonic schemes, using an $\mathbf{ss}$ standard for one and a $\mathbf{tt}$ standard for the other.
Finally, we expand on the abstract relations between the transitive closure and dual transitive closure operations, showing that they give rise to lattice operations that precisely capture how the logics discussed relate to one another.
\end{abstract}


\section{Introduction}
Are there three-valued presentations of classical logic? \cite{Cobrerosetal2012} answer this question affirmatively by showing that certain three-valued logics based on the Strong Kleene scheme validate precisely the same inferences as classical logic. These results were later on extended to the Weak Kleene counterparts of these logics by \cite{szmucferguson2021} and \cite{ferguson2023monstrous}. Ultimately, \cite{da2023three} identified all the three-valued truth-functional logics corresponding to classical logic. In particular, they identified a subclass of these logics, characterized by a strict-tolerant ($\mathbf{st}$) standard of argument evaluation, defined over a \textit{Boolean normal monotonic scheme} for the connectives (BNM). One such logic is the strict-tolerant logic $\mathsf{ST}$ \citep{Cobrerosetal2012, ripley2012conservatively, Cobrerosetal2013}, based on the aforementioned Strong Kleene scheme. They further proved that logics defined with a standard among $\mathbf{ss}$, $\mathbf{tt}$, $\mathbf{ts}$ and $\mathbf{ss} \cap \mathbf{tt}$ over a BNM scheme are necessarily weaker than classical logic. The Strong Kleene logic $\mathsf{K}_3$ \citep{Kleene1952-KLEITM}, the logic of paradox $\mathsf{LP}$ \citep{asenjo1966, priest1979}, the tolerant-strict logic $\mathsf{TS}$ \citep{Cobrerosetal2012}, and the logic of order $\mathsf{KO}$ \citep{Makinson1973} are notable examples of such logics, all based on the Strong Kleene scheme. Other such logics include those based on the Middle Kleene scheme \citep{Peters1979, BeaverKrahmer2001, George2014}, and the logics based on the Weak Kleene scheme, such as \cite{Bochvar1938}'s logic $\mathsf{K}_3^w$ and \cite{Hallden1949}'s Paraconsistent Weak Kleene logic $\mathsf{PWK}$.

Nonetheless, the authors do not investigate the properties of logics defined by the union of the $\mathbf{ss}$ and $\mathbf{tt}$ standards. It is known, for instance, that the union of $\mathsf{K}_3$ and $\mathsf{LP}$ does not align with classical logic (see e.g.\ \citealp[p.511]{wintein2016all}) or equivalently $\mathsf{ST}$. The relationship between $\mathsf{K}_3$ and $\mathsf{LP}$ to classical logic has been studied by \cite{blomet2024sttsproductsum}, where it is shown that the set of valid inferences of $\mathsf{ST}$ is the relational composition of the sets of $\mathsf{K}_3$- and $\mathsf{LP}$-valid inferences. However, the authors do not examine the relationship between other logics based on the $\textbf{ss}$ and $\textbf{tt}$ standards and classical logic.

This paper first extends the aforementioned inequality result, showing that
no logic defined by the union of the $\mathbf{ss}$ and $\mathbf{tt}$ standards and a BNM scheme coincides with classical logic. This naturally raises the question: what is absent from this union to recover classical logic? We will address this by showing that what is missing is \textit{transitivity}. Once closed under transitivity, the union of an $\mathbf{ss}$- and a $\mathbf{tt}$-logic is classical logic.

Just as the union of an $\mathbf{ss}$- and a $\mathbf{tt}$-logic based on a BNM scheme fails to be classical logic, the intersection of an $\mathbf{ss}$- and a $\mathbf{tt}$-logic does not coincide with the aforementioned non-reflexive logic $\mathsf{TS}$. \cite{da2023three} show in particular that any logic with a $\mathbf{ts}$ standard defined over a BNM scheme collapses with $\mathsf{TS}$. In contrast, the intersection of an $\mathbf{ss}$- and a $\mathbf{tt}$-logic is reflexive, as exemplified by the logic $\mathsf{KO}$. We are thus justified in asking a similar question to the one posed for the union: what is missing from the intersection to yield the logic $\mathsf{TS}$? Once again, we will answer this by showing that what is missing is a relational property---specifically the property of \textit{dual transitivity}.

We will then explore the relationship between the transitive closure operator and its dual, as well as their effects when applied to sets of validities. This analysis will ultimately enable us to present a cohesive overview of the interrelations among all the logics discussed, showing that their sets of valid inferences form a sublattice of the lattice of logics satisfying the Tarskian properties of Reflexivity, Monotonicity, Transitivity, and Structurality. We further show that their sets of anti-theorems and theorems---defined here as inferences with unsatisfiable premises and inferences with an unfalsifiable conclusion---form a lattice induced by the dual transitive closure operator.

The paper is structured as follows. In Section \ref{sect:def}, we provide all the preliminary technicalities required for the paper. In Section \ref{sct:TrDef}, we introduce the operations of \textit{transitive closure} and \textit{dual transitive closure}, and present their main properties. In Section \ref{c} we prove that the transitive closure of the union of two $\mathbf{ss}$- and $\mathbf{tt}$-logics coincides with an $\mathbf{st}$-logic (that is, with classical logic), based on possibly different BNM schemes. We then dualize this result and prove that applying the dual transitive closure to the intersection of an $\mathbf{ss}$- and a $\mathbf{tt}$-logic yields a $\mathbf{ts}$-logic (i.e., the empty logic), possibly based on different BNM schemes. In Section \ref{sec:aditional}, we elaborate on some remarks regarding the dual transitive closure operator. In Section \ref{sec:lattice}, we provide an overview of the interrelations among all the logics discussed. Finally, in Section \ref{sect:conclusion}, we conclude with our final observations and remarks.

\section{Preliminary definitions}\label{sect:def}

\begin{Definition}[Language]\label{language2b}
The propositional language $\mathcal{L}$ is built from a denumerably infinite set $\textit{Var} = \lbrace p, p', \ldots\rbrace$ of propositional variables, using the logical constants negation ($\neg$), disjunction ($\lor$) and conjunction ($\wedge$).
\end{Definition}

\noindent
Elements of $\mathcal{L}$ will be denoted by lower-case Greek letters or upper-case Latin letters, depending on the context. Subsets of $\mathcal{L}$ will be denoted by upper-case Greek letters. Given a function $f: \mathit{Var} \longrightarrow \mathcal{L}$, $f$ extends to a \textit{substitution} $\sigma$ on $\mathcal{L}$ by letting $\sigma (\star (\phi_1 \ldots \phi_n)) = \star (\sigma (\phi_1) \ldots \sigma (\phi_n))$ for all logical constants $\star$. Given a formula $\phi$, $\textup{At}(\phi)$ denotes the set of propositional variables of $\phi$. By extension, given a set of formulas $\Gamma$, $\textup{At}(\Gamma) \coloneqq \bigcup \lbrace \textup{At}(\gamma) : \gamma \in \Gamma \rbrace$.

\begin{Definition} 
An \textit{inference} is an ordered pair $\langle \Gamma,\phi \rangle $, denoted by $\Gamma \Rightarrow \phi$, where $\phi \in \mathcal{L}$ and $\Gamma \subseteq \mathcal{L}$ is finite.
\end{Definition}

Before proceeding, a few remarks on the previous definition are in order. First, note that inferences are defined with a finite set of premises. However, this does not compromise the generality of the results in this paper. All the logics discussed here are finitely-valued, and therefore finitary (see \citealp{Wojcicki1988} 4.1.7). Consequently, they are entirely determined by their set of inferences with finitely many premises. Second, we choose to work with single conclusions, as this framework allows for simpler and more elegant definitions of the operations introduced in the next section.\footnote{Note that the authors of \cite{da2023three} adopt a multiple-conclusion setting. However, none of their results fundamentally depend on this choice, as everything in that paper could be adapted to a single-conclusion framework. We intend to address the multiple-conclusion case in future work.} 

Blackboard letters
$\mathbb{L}$ are used to denote sets of inferences. Logics $\mathsf{L}$ are occasionally identified with sets
of inferences $\mathbb{L}$.\footnote{Although extensionally identical, we distinguish between these two objects. This distinction is necessary since we sometimes refer to logical properties (such as the validity of an inference), while at other times, we consider set-theoretic properties (such as the union of two sets of inferences).} Such an identification is more liberal than the common definition of a logic as a set of inferences satisfying the Tarskian properties of Reflexivity, Monotonicity, Transitivity, and Structurality (see, e.g. \citealp{font2016}). Recognizing non-transitive systems like $\mathsf{ST}$ or non-reflexive ones like $\mathsf{TS}$ as logics requires adopting a broader definition. We therefore distinguish a logic simpliciter from a Tarskian logic.

 \TabPositions{1cm}
\begin{Definition}[Tarskian logic]
    A logic $\mathsf{L}$ is said to be \textit{Tarskian} if it satisfies the following properties, for all $\Gamma, \Sigma, \lbrace \phi \rbrace \subseteq \mathcal{L}$.
    \begin{enumerate}[wide=0pt, align=left]
    
        \item[(R)] \tab $\phi \Rightarrow \phi \in \mathbb{L}$. \hspace*{0pt}\hfill (Reflexivity)\label{pr:id}
        \item[(M)] \tab If $ \Gamma \Rightarrow \phi \in \mathbb{L}$ and $\Gamma \subseteq \Sigma$, then $ \Sigma \Rightarrow \phi \in \mathbb{L}$. \hspace*{0pt}\hfill 
        (Monotonicity)\label{pr:mo}
        \item[(T)] \tab $\text{If} \ (\exists\Delta \neq \emptyset) \ \Delta \Rightarrow \phi \in \mathbb{L} \ \text{and} \ (\forall\delta \in \Delta) \ \Gamma \Rightarrow \delta \in \mathbb{L}$,\\ 
         \hphantom{\hspace{0.86cm}} $\text{then} \ \Gamma \Rightarrow \phi \in \mathbb{L}.$ \hspace*{0pt}\hfill 
        (Transitivity)\label{pr:tr}
        \item[(S)] \tab If $ \Gamma \Rightarrow \phi \in \mathbb{L}$, then $ \sigma[\Gamma] \Rightarrow \sigma[\phi] \in \mathbb{L}$ for all substitutions $\sigma$. \hspace*{0pt}\hfill 
        (Structurality)\label{pr:st}
    \end{enumerate}
\end{Definition}

A Tarskian logic can be obtained from any set of inferences by closing it under (R), (M), (T), and (S).

\begin{Definition}[Tarskian closure]
For any set of inferences $\mathbb{L}$, the Tarskian closure of $\mathbb{L}$, noted $Tar(\mathbb{L})$, is the least set $\mathbb{L}' \supseteq \mathbb{L}$ satisfying (R), (M), (T) and (S).
\end{Definition}

Given a logic, we provide a semantic characterization of its set of valid inferences. 
A \textit{valuation} is a function $v$ from $\mathcal{L}$ to $\mathcal{V}$, where $\mathcal{V}$ is a set of truth values. It is important to note that, as for now, valuations are not subject to any constraints of truth-functionality. In this paper, we will consider two possible sets of values $\mathcal{V}_{2}=\{\f,\t\}$ and $\mathcal{V}_{3}=\{\f,\i, \t\}$. 

\begin{Definition}
 A \textit{formula-standard} $\bf{z}$ is defined as $\bf{z} \subseteq \mathcal{V}$. A \textit{standard} is defined as $\bf{z}=\bf{xy}$, where $x,y$ are formula-standards.
\end{Definition}

\begin{Definition}\label{def:valuation}
   A valuation $v$ satisfies a formula $\phi$ according to a formula-standard $\bf{x}$ ($v\vDash_{\bf{x}} \phi$), if $v(\phi) \in \bf{x}$. A valuation $v$ satisfies an inference $\Gamma\Rightarrow\phi$ according to a standard $\bf{xy}$, denoted as $v\vDash_{\bf{xy}}\Gamma\Rightarrow\phi$, if the following holds:
   \[\text{if} \ v\vDash_{\bf{x}}\gamma \ \text{for all} \ \gamma \in \Gamma, \ \text{then} \ v\vDash_{\bf{y}}\phi.\]
   
\noindent
An inference is valid according to a standard $\bf{xy}$, denoted as $\vDash_{\bf{xy}}\Gamma\Rightarrow\phi$, if $v \vDash_{\bf{xy}}\Gamma\Rightarrow\phi$ for all $v$.
\end{Definition}

Two well-studied formula-standards on $\mathcal{V}_{3}$ are the \textit{strict} and the \textit{tolerant} standards, defined respectively as $\bf{s}=\{1\}$ and $\bf{t}=\{1,\i\}$. With these two standards, it is possible to define four possible standards: $\bf{ss},\bf{tt},\bf{st}$ and $\bf{ts}$ (strict-strict, tolerant-tolerant, strict-tolerant and tolerant strict, respectively). Given a standard $\bf{xy}$ 
we can define a set of inferences $\mathbb{XY}$ containing exactly the inferences valid according to $\bf{xy}$ by every valuation. In particular, from the standards, we can extract six sets of inferences, which are ordered by inclusion in Figure \ref{consequence-relations} \citep{Chemlaetal2017}.

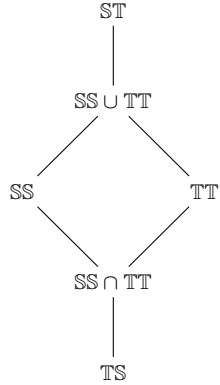
\begin{figure}[ht]
\centering
\begin{tikzpicture}[scale=0.6]
{\footnotesize
  \node (max) at (0,4) {$\mathbb{ST}$};
  \node (b) at (0,2)  {$\mathbb{SS} \cup \mathbb{TT}$};
  \node (a) at (-2,0) {$\mathbb{SS}$};
  \node (c) at (2, 0) {$\mathbb{TT}$};
  \node (d) at (0,-2) {$\mathbb{SS} \cap \mathbb{TT}$};
  \node (min) at (0,-4) {$\mathbb{TS}$};
  
  \draw (max) -- (b);
  \draw (b) -- (a);
  \draw (b) -- (c);
  \draw (a) -- (d);
  \draw (c) -- (d);
  \draw (d) -- (min);
}
\end{tikzpicture}
\caption{Six sets of inferences extracted from the standards $\bf{s},\bf{t}$}
\label{consequence-relations}
\end{figure}

Up to this point, we have imposed no constraints on the trivaluations. In this article, we will consider valuations respecting some three-valued \textit{scheme} $\X$, where  $\X$ is a triple $(f_\neg, f_\wedge, f_\vee)$ of operations. In particular, following \cite{da2023three}, we will consider \textit{Boolean normal monotonic} schemes (hereinafter referred to as BNM schemes).

\begin{Definition}[\citealp{da2023three}]
\label{def:boolean-normality2}
An $n$-ary operation $\star$ is \textit{Boolean normal} if and only if for $\{a_1,\ldots,a_n\} \subseteq \{0,1\}$, $\star(a_1,\ldots,a_n)=\star^{\CL}(a_1,\ldots,a_n)$, where $\star^{\CL}$ is the corresponding operation over the usual two-elements Boolean algebra. 
A scheme is \textit{Boolean normal} if and only if each of its operations is.
\end{Definition}

\noindent
When considering the Boolean normal bivaluations on $\mathcal{V}_{2}$, all the standards define what we refer to as the set of classical inferences, denoted $\mathbb{CL}$. 

For \textit{monotonicity}, we first assume that the truth values are ordered according to the order defined as: $\sfrac{1}{2} <_{\rm I} 0$ and $\sfrac{1}{2} <_{\rm I} 1$. The componentwise ordering induced by this relation is then defined as follows: $\langle a_1,\ldots,a_n \rangle \leq_{\rm I}^{comp} \langle b_1,\ldots,b_n \rangle$ if and only if $a_j \leq_{\rm I} b_j$ for all $1 \leq j \leq n$.

\begin{Definition}[\citealp{da2023three}]
\label{def:monotonicity2}
An $n$-ary operation $\star$ is \textit{monotonic} if and only if whenever $\langle a_1,\ldots,a_n \rangle \leq_{\rm I}^{comp} \langle b_1,\ldots,b_n \rangle$ then $\star(a_1,\ldots,a_n) \leq_{\rm I} \star(b_1,\ldots,b_n)$. 
A scheme is monotonic if and only if each of its operations is.
\end{Definition}

A valuation $v$ is said to respect a scheme $\bf{X}$ if for every operation $f^{\star}$ of $\bf{X}$ and corresponding symbol $\star$ in the language, it satisfies
\[v(\star (\phi_1 \ldots \phi_n)) = f^{\star} (v(\phi_1) \ldots v(\phi_n)).\]
\noindent
We now observe that a set of valuations, together with a standard, fully determines the semantics of a logic.

Occasionally, when working with a logic defined by some standard $\bf{xy}$ and a scheme $\mathbf{X}$, we denote the fact that an inference $\Gamma \Rightarrow \phi$ is satisfied by writing $v \vDash_{\bf{xy}}^{\mathbf{X}} \Gamma \Rightarrow \phi$. We say that $\Gamma \Rightarrow \phi$ is valid in $\mathsf{L}$ if it is satisfied by every valuation (symbolically, $\vDash_{\bf{xy}}^{\mathbf{X}} \Gamma \Rightarrow \phi$). When the context makes the scheme clear, we omit the superscript $\bf{X}$.

\section{Transitive closure and dual transitive closure}\label{sct:TrDef}
In this section, we introduce two operators: the \textit{transitive closure operator} and the \textit{dual transitive closure operator}. The transitive closure operator will enable us to define the set of valid inferences of an $\bf{st}$ logic based on a BNM scheme. This will be achieved by closing under transitivity the union of the sets of valid inferences of two $\bf{ss}$ and $\bf{tt}$ logics, based on a BNM scheme. In contrast, the dual transitive operator will allow us to define the set of valid inferences of a $\bf{ts}$ logic. Specifically, applying such operation to the intersection of the sets of valid inferences of two $\bf{ss}$ and $\bf{tt}$ logics will yield the sets of valid inferences of a $\bf{ts}$ logic.



\begin{Definition}[Transitive closure]
    For any set of inferences $\mathbb{L}$, the transitive closure of $\mathbb{L}$, noted $T(\mathbb{L})$, is the least set $\mathbb{L}' \supseteq \mathbb{L}$ such that: 
    $$ \text{If} \ (\exists\Delta \neq \emptyset) \ \Delta \Rightarrow \phi \in \mathbb{L}' \ \text{and} \ (\forall\delta \in \Delta) \ \Gamma \Rightarrow \delta \in \mathbb{L}', \ \text{then} \ \Gamma \Rightarrow \phi \in \mathbb{L}'.$$
\end{Definition}

\noindent
We need to ensure that the operator is a closure operator.\footnote{A function $C: \mathcal{P}(A) \to \mathcal{P}(A)$ is said to be a closure operator on $A$ if
\begin{itemize}
\itemsep0em
    \item[--] For all $X \subseteq A$, $X \subseteq C(X)$,
    \item[--] For all $X, Y \subseteq A$, if $X \subseteq Y$, then $C(X) \subseteq C(Y)$,
    \item[--] For all $X \subseteq A$, $CC(X)=C(X)$.
\end{itemize}} This is guaranteed by the following fact.

\begin{fact}
$T$ is a closure operator.
\end{fact}
\begin{proof}
\quad
    \begin{itemize}
        \item[--] The inclusion $\mathbb{L} \subseteq T(\mathbb{L})$ for all $\mathbb{L}$ follows directly from the definition of transitive closure.

    \item[--] Turning to the monotonicity of $T$, assume that $\mathbb{L} \subseteq \mathbb{L}'$. Then $\mathbb{L} \subseteq \mathbb{L}' \subseteq T(\mathbb{L}')$. So, $T(\mathbb{L}')$ is an extension of $\mathbb{L}$ such that for all $\Delta \neq \emptyset, \text{if} \ \Delta \Rightarrow \phi \in T(\mathbb{L}') \ \text{and} \ \Gamma \Rightarrow \delta \in T(\mathbb{L}') \ \text{for all} \ \delta \in \Delta, \ \text{then} \ \Gamma \Rightarrow \phi \in T(\mathbb{L}')$. By definition, $T(\mathbb{L})$ is the least extension of $\mathbb{L}$ satisfying such a property, so $T(\mathbb{L}) \subseteq T(\mathbb{L}')$.


    \medskip
    
   \item[--] To show that $T(T(\mathbb{L}))=T(\mathbb{L})$, it suffices to establish the left-to-right inclusion, as the reverse direction follows directly from the previously proven fact that $\mathbb{L} \subseteq T(\mathbb{L})$ for all $\mathbb{L}$. 
   Trivially, $T(\mathbb{L})$ is an extension of itself and is such that for all $\Delta \neq \emptyset, \text{if} \ \Delta \Rightarrow \phi \in T(\mathbb{L}) \ \text{and} \ \Gamma \Rightarrow \delta \in T(\mathbb{L}) \ \text{for all} \ \delta \in \Delta, \ \text{then} \ \Gamma \Rightarrow \phi \in T(\mathbb{L})$. Now, by definition, $T(T(\mathbb{L}))$ is the least extension of $T(\mathbb{L})$ with such a property, hence $T(T(\mathbb{L})) \subseteq T(\mathbb{L})$.
    \end{itemize}
\end{proof}

\begin{Definition}[Dual transitive closure]
    For any set of inferences $\mathbb{L}$, the dual transitive closure of $\mathbb{L}$, noted $T^d(\mathbb{L})$, is the greatest set $\mathbb{L}' \subseteq \mathbb{L}$ such that:
    $$ \text{If} \ (\exists\Delta \neq \emptyset) \ \Delta \Rightarrow \phi \not\in \mathbb{L}' \ \text{and} \ (\forall\delta \in \Delta) \ \Gamma \Rightarrow \delta \not\in \mathbb{L}', \ \text{then} \ \Gamma \Rightarrow \phi \not\in \mathbb{L}'.$$
\end{Definition}

Next, we show in what sense this notion of dual transitive closure is dual to the notion of transitive closure. The proposition hereinafter eventually expresses the duality of the transitive closure operator and the dual transitive closure operator. The two operators are interdefinable through the set-theoretical notion of absolute complement.

\begin{proposition}\label{prp:duaTr}
    \begin{align*}
    T^{d}(\mathbb{L}) &= \overline{T(\overline{\mathbb{L}})}\\
    T(\mathbb{L}) &= \overline{T^{d}(\overline{\mathbb{L}})}.
    \end{align*}
\end{proposition}

\begin{proof}
For any set of inferences $\mathbb{L}$, let $P(\mathbb{L})$ express that if $(\exists\Delta \neq \emptyset) \ \Delta \Rightarrow \phi \in \mathbb{L}$ and $(\forall\delta \in \Delta) \ \Gamma \Rightarrow \delta \in \mathbb{L}$, it follows that $\Gamma \Rightarrow \phi \in \mathbb{L}$.  

Starting with the first identity, by definition, $T(\overline{\mathbb{L}})$ is the least element of the set  
\[
\lbrace \mathbb{L}' \mid \overline{\mathbb{L}} \subseteq \mathbb{L}', \ P(\mathbb{L}') \rbrace.
\]  
Equivalently, $\overline{T(\overline{\mathbb{L}})}$ is the greatest element of  
\[
\lbrace \overline{\mathbb{L}'} \mid \overline{\mathbb{L}} \subseteq \mathbb{L}', \ P(\mathbb{L}') \rbrace,
\]  
which coincides with  
\[
\lbrace \mathbb{L}'' \mid \mathbb{L}'' \subseteq \mathbb{L}, \ P(\overline{\mathbb{L}''}) \rbrace.
\]  
That is, $\overline{T(\overline{\mathbb{L}})}$ is the largest subset $\mathbb{L}''$ of $\mathbb{L}$ such that if $(\exists\Delta \neq \emptyset) \ \Delta \Rightarrow \phi \not\in \mathbb{L}''$ and $(\forall\delta \in \Delta) \ \Gamma \Rightarrow \delta \not\in \mathbb{L}''$, then $\Gamma \Rightarrow \phi \not\in \mathbb{L}''$.
Therefore, we conclude that $T^d(\mathbb{L}) = \overline{T(\overline{\mathbb{L}})}$. The other identity can be proved similarly.
\end{proof}

\enlargethispage{15pt}

It can now be checked that $T^d$ is an \textit{interior operator}.\footnote{A function $I: \mathcal{P}(A) \to \mathcal{P}(A)$ is an interior operator if it satisfies
\begin{itemize}
\itemsep0em
    \item[--] For all $X \subseteq A$, $I(X) \subseteq X$,
    \item[--] For all $X, Y \subseteq A$, if $X \subseteq Y$, then $I(X) \subseteq I(Y)$,
    \item[--] For all $X \subseteq A$, $II(X)=I(X)$.
\end{itemize}}

\begin{fact}
    $T^{d}$ is an interior operator.
\end{fact}
\begin{proof}
\quad
    \begin{itemize}
        \item[--] Since $T$ is a closure operator, $\overline{\mathbb{L}} \subseteq T(\overline{\mathbb{L}})$, and thus $T^{d}(\mathbb{L}) = \overline {T(\overline{\mathbb{L}})} \subseteq \overline{\overline{\mathbb{L}}} = \mathbb{L}$ by Proposition \ref{prp:duaTr}.
        \item[--] Assume $\mathbb{L} \subseteq \mathbb{L}'$. Then $\overline{\mathbb{L}'} \subseteq \overline{\mathbb{L}}$, so $T(\overline{\mathbb{L}'}) \subseteq T(\overline{\mathbb{L}})$ since $T$ is a closure operator, and hence $T^{d}(\mathbb{L}) = \overline{T(\overline{\mathbb{L}})} \subseteq \overline{T(\overline{\mathbb{L}'})} = T^{d}(\mathbb{L}')$ by Proposition \ref{prp:duaTr}.
        \item[--] Since $T$ is a closure operator, $T(\overline{\mathbb{L}}) \subseteq T(T(\overline{\mathbb{L}}))$, so $\overline{T(T(\overline{\mathbb{L}}))} \subseteq \overline{T(\overline{\mathbb{L}})}$. But by Proposition \ref{prp:duaTr}, $T^{d}(T^{d}(\mathbb{L})) = \overline{T(\overline{T^{d}(\mathbb{L})})} = \overline{T(T(\overline{\mathbb{L}}))}$, therefore $T^{d}(T^{d}(\mathbb{L})) \subseteq \overline{T(\overline{\mathbb{L}})} = T^{d}(\mathbb{L})$. 
    \end{itemize}
\end{proof}

With these definitions and facts at our disposal, we are equipped to offer an exact characterization of $\mathbb{ST}$ as the transitive closure of $\mathbb{SS} \cup \mathbb{TT}$, and $\mathbb{TS}$ as the dual transitive closure of $\mathbb{SS} \cap \mathbb{TT}$.

\section{Characterization of $\mathbb{ST}$ and $\mathbb{TS}$}\label{c}

Our aim in this section is to characterize $\mathbf{st}$- and $\mathbf{ts}$-logics by combining $\mathbf{ss}$- and $\mathbf{tt}$-logics. We first establish that the union of $\mathbb{SS}$ and $\mathbb{TT}$ does not yield $\mathbb{ST}$. We then prove that the transitive closure of their union actually coincides with $\mathbb{ST}$ (and hence with $\mathbb{CL}$). A dual version of this result is then proved for the characterization of $\mathbf{ts}$-logics: their set of valid inferences is obtained by applying the dual transitive closure to the intersection of the sets of $\mathbf{tt}$- and $\mathbf{ss}$-valid inferences.
Importantly, all the results in this section assume that the sets of inferences are defined using \textit{BNM} schemes, possibly different from one logic to another. We start by restating the following results  proven in \cite{da2023three}:

\begin{theorem}[\citealp{da2023three}]\label{thm:cl=st}
For every BNM scheme, $\mathbb{ST} = \mathbb{CL}$.
\end{theorem}

\begin{theorem}[\citealp{da2023three}]\label{thm:empty=st}
For every BNM scheme, $\mathbb{TS} = \emptyset$.
\end{theorem}

The first theorem shows that any $\bf{st}$-logic based on a BNM scheme collapses with classical logic. The second one that any $\bf{ts}$-logic based on a BNM scheme is empty (i.e. invalidate every inference). 

We now show that the union of the sets of inferences $\mathbb{SS}$ and $\mathbb{TT}$, obtained from the $\mathbf{ss}$ and $\mathbf{tt}$ standards, differs from $\mathbb{ST}$. To establish this, we provide a counterexample: an inference that is $\mathsf{ST}$-valid but neither $\mathsf{SS}$- nor $\mathsf{TT}$-valid. First, we state the following relation between these sets.

\begin{fact}\label{fct:subs}
    Let $\mathbb{SS}, \mathbb{TT}$ and $\mathbb{ST}$ be defined by (possibly different) BNM schemes. Then
    \begin{itemize}
    \itemsep0em
        \item $\mathbb{SS}, \mathbb{TT} \subseteq \mathbb{ST}$,
        \item $\mathbb{SS}$ and $\mathbb{TT}$ are incomparable.
    \end{itemize}
\end{fact}
\begin{proof}
For the first bullet, take any BNM-valuation such that $v(p)=\sfrac{1}{2}$ and $v(r)=0$. It's easy to notice that $v\nvDash_{\bf{tt}} p \wedge \neg p \Rightarrow r$ and $v\nvDash_{\bf{ss}} \Rightarrow p \vee \neg p $.\footnote{Notice that for every BNM scheme, if $v(p)=\sfrac{1}{2}$ then $v(\neg p)=\sfrac{1}{2}$, and if $v(p)=v(q)=\sfrac{1}{2}$, then $v(p\wedge q)=v(p\vee q)=\sfrac{1}{2}$.} However, both inferences are classically valid, and thus by Theorem \ref{thm:cl=st} they are in $\mathbb{ST}$.


As for the incomparability of $\mathbb{SS}$ and $\mathbb{TT}$, just note that, for any BNM scheme $\mathbf{X}$, $\not\models^{\mathbf{X}}_{\bf{tt}} p \wedge \neg p \Rightarrow r $, but $\models^{\mathbf{X}}_{\bf{ss}} p \wedge \neg p \Rightarrow r$, and $\not\models^{\mathbf{X}}_{\bf{ss}} \emptyset \Rightarrow p \lor \neg p$, but $\models^{\mathbf{X}}_{\bf{tt}} \emptyset \Rightarrow p \lor \neg p$.
\end{proof}

\begin{theorem}\label{th:ssunionttinclcl}
$\mathbb{SS} \cup \mathbb{TT} \subsetneq \mathbb{ST}=\mathbb{CL}.$

\end{theorem}

\begin{proof}
By Theorem \ref{thm:cl=st}, we know that $\mathbb{ST}=\mathbb{CL}$. The fact that $\mathbb{SS} \cup \mathbb{TT} \subseteq \mathbb{ST}$ is trivial given the previous fact. So, let us show that $\mathbb{SS} \cup \mathbb{TT} \neq \mathbb{ST}.$ In order to do so, let us consider $ p \vee (q \wedge \neg q) \Rightarrow p \wedge (r \vee \neg r)$. This inference is the witness of the claim. 

Note first that given any BNM scheme, there is a valuation $v$ such that $v \nvDash_{\bf{ss}} p \vee (q \wedge \neg q) \Rightarrow p \wedge (r \vee \neg r)$ (e.g.\ $v(p)=1,v(q)=0,v(r)=\sfrac{1}{2}$) and there is another valuation $v'$ such that $v'\nvDash_{\bf{tt}} p \vee (q \wedge \neg q) \Rightarrow p \wedge (r \vee \neg r)$ (e.g.\ $v'(p)=0,v'(r)=0,v'(q)=\sfrac{1}{2}$). However, $ \vDash_{\bf{st}} p \vee (q \wedge \neg q) \Rightarrow p \wedge (r \vee \neg r)$ since the inference is classically valid. 
\end{proof}

One may then ask whether the gap between the union of $\mathbb{SS}$ and $\mathbb{TT}$ logics and $\mathbb{ST}$ (i.e. the set of classical inferences) can be bridged. The following result shows that the transitive closure is the key to obtaining $\mathbb{ST}$ from the union of $\mathbb{SS}$ and $\mathbb{TT}$.

\begin{theorem}\label{th:st0} 
$$T(\mathbb{SS} \cup \mathbb{TT}) = \mathbb{ST}.$$ 
    
\end{theorem}

\begin{proof}
    ($\subseteq$) From Fact \ref{fct:subs}, it holds that $\mathbb{SS}\cup \mathbb{TT} \subseteq \mathbb{ST}$, so $T(\mathbb{SS}\cup \mathbb{TT})\subseteq T(\mathbb{ST})$ since $T$ is a closure operator. Now, $T(\mathbb{ST}) = \mathbb{ST}$, and therefore $T(\mathbb{SS}\cup \mathbb{TT})\subseteq \mathbb{ST}$.
    
    \medskip

   

    ($\supseteq$) Assume $\Gamma \Rightarrow \phi \in \mathbb{ST}$, that is $\models_{\bf{st}}^{\mathbf{X}} \Gamma \Rightarrow \phi$. We show in particular that there is $\Delta \neq \emptyset$ such that $\models_{\bf{tt}}^{\mathbf{X}'} \Gamma \Rightarrow \delta$ for all $\delta \in \Delta$ and $\models_{\bf{ss}}^{\mathbf{X}''} \Delta \Rightarrow \phi$ with $\mathbf{X}'$ and $\mathbf{X}''$ two BNM schemes possibly different from $\mathbf{X}$. Let $\Delta$ be as follows:

\begin{equation*}
\Delta \coloneqq \Gamma \cup \lbrace p \lor \neg p : p \in \textup{At}(\phi) \rbrace
\end{equation*}

\noindent
For all $\gamma \in \Gamma$, $\models_{\bf{tt}}^{\mathbf{X}'} \Gamma \Rightarrow \gamma$, since $\mathbb{TT}$ is reflexive. Moreover, 
$\models_{\bf{tt}}^{\mathbf{X}'} \Gamma \Rightarrow p \lor \neg p$ for all $p \in \textup{At}(\phi)$, since for any BNM valuation, $v(p \lor \neg p) \neq 0$. Turning to $\models_{\bf{ss}}^{\mathbf{X}''} \Delta \Rightarrow \phi$, assume $v(\delta)= 1$ for all $\delta \in \Delta$. By definition of $\Delta$, it follows that $v(p \lor \neg p)=1$ for all $p \in \textup{At}(\phi)$, and given that ${\mathbf{X}''}$ is a BNM scheme, $v(p) \neq \sfrac{1}{2}$ for all $p \in \textup{At}(\phi)$. In addition, $v(\gamma) =1$ for all $\gamma \in \Gamma$ and  $\models_{\bf{st}}^{\mathbf{X}} \Gamma \Rightarrow \phi$ by assumption. By Theorem \ref{thm:cl=st}, $\models_{\bf{st}}^{\mathbf{X}''} \Gamma \Rightarrow \phi$, so $v(\phi) \neq 0$. But $v(p) \neq \sfrac{1}{2}$ for all $p \in \textup{At}(\phi)$, thus $v(\phi) =1$, since $v$ is a BNM valuation. Now, $\Gamma \Rightarrow \delta \in \mathbb{TT}$ for all $\delta \in \Delta$ and $\Delta \Rightarrow \phi \in \mathbb{SS}$, meaning that $\Gamma \Rightarrow \delta \in \mathbb{SS} \cup \mathbb{TT}$ for all $\delta \in \Delta$ and $\Delta \Rightarrow \phi \in \mathbb{SS} \cup \mathbb{TT}$. Therefore, $\Gamma \Rightarrow \phi \in T(\mathbb{SS} \cup \mathbb{TT})$.

\end{proof}


Note that the previous results rely heavily on the absence of truth constants for the intermediate value in the language. If we introduce a constant for this value, say $\lambda$, the closure under transitivity of the union of the valid inferences from any $\mathbf{tt}$-and $\mathbf{ss}$-logic would result in a trivial set of inferences. To see this, just note that for all $\Gamma, \phi$, it holds that $\Gamma \Rightarrow \lambda \in \mathbb{TT}$ and $\lambda \Rightarrow \phi \in \mathbb{SS}$. As a result, $\Gamma \Rightarrow \phi \in T(\mathbb{SS} \cup \mathbb{TT})$ for any inference $\Gamma \Rightarrow \phi$, collapsing $\mathsf{ST}$ with the trivial logic, that is, the logic corresponding to the universal relation on $\mathcal{L}$.\footnote{For characterizations of $\mathsf{ST}$ in languages containing $\lambda$, see \cite{blomet2024sttsproductsum}.}

Let us now turn our attention to the case of $\bf{ts}$-logics. It is straightforward to notice that $\mathbb{TS}\subsetneq \mathbb{TT} \cap \mathbb{SS}$, since $\mathbb{TS}=\emptyset$ while $\mathbb{TT} \cap \mathbb{SS}\neq \emptyset$ (for instance $\phi \Rightarrow \phi \in \mathbb{TT} \cap \mathbb{SS}$, for any $\phi$). Thus to collapse $\mathbb{TT} \cap \mathbb{SS}$ with $\mathbb{TS}$ we need to erase all the inferences belonging to this set. The next result shows that the dual transitive closure operator does this job. 

\begin{theorem}\label{thm:tssum}
    \begin{align*}
        T^d(\mathbb{TT} \cap \mathbb{SS}) = \mathbb{TS}
    \end{align*}
\end{theorem}

\begin{proof}
($\subseteq$) Let $\Gamma \Rightarrow \phi \in T^d(\mathbb{TT} \cap \mathbb{SS})$. Then for all $\Delta \neq \emptyset$, $\Gamma \Rightarrow \delta \in \mathbb{TT} \cap \mathbb{SS}$ for some $\delta \in \Delta$ or $\Delta \Rightarrow \phi \in \mathbb{TT} \cap \mathbb{SS}$. Meaning that for all $\Delta \neq \emptyset$, $\models_{\mathbf{tt}} \Gamma \Rightarrow \delta$ for some $\delta \in \Delta$ or $\models_{\mathbf{ss}} \Delta \Rightarrow \phi$, and in particular that for some $p \not\in \textup{At}(\Gamma) \cup \textup{At}(\phi)$, $\models_{\mathbf{tt}} \Gamma \Rightarrow p$ or $\models_{\mathbf{ss}} p \Rightarrow \phi$. Since $\Gamma$ is finite, such a $p$ must exist. Now, $\models_{\mathbf{tt}} \Gamma \Rightarrow p$ is impossible, as for any BNM scheme, there exists a valuation $v$ such that $v(p) = 0$ and $v(q) = \sfrac{1}{2}$ for all $q \in \mathit{Var}$ with $q \neq p$, ensuring that $v \not\models_{\mathbf{tt}} \Gamma \Rightarrow p$. Likewise, $\models_{\mathbf{ss}} p \Rightarrow \phi$ is also impossible, since for any BNM scheme, there exists a valuation $v'$ such that $v'(p) = 1$ and $v'(q) = \sfrac{1}{2}$ for all $q \in \mathit{Var}$ with $q \neq p$, ensuring that $v \not\models_{\mathbf{ss}} p \Rightarrow \phi$. Therefore, $T^d(\mathbb{TT} \cap \mathbb{SS}) = \emptyset = \mathbb{TS}$.

\medskip
\noindent
 $(\supseteq)$
Trivial, since $\mathbb{TS}=\emptyset$.
\end{proof}

Having established how to obtain $\mathbb{ST}$ from $\mathbb{SS} \cup \mathbb{TT}$, and $\mathbb{TS}$ from $\mathbb{SS} \cap \mathbb{TT}$, we will now focus on the logics corresponding to these sets. First, note that all the results from the previous sections were stated for any BNM scheme. However, it is not necessary for the scheme used to define $\mathbb{SS}$ and $\mathbb{TT}$ to be the same. For example, we can consider the two well-known schemes from Figure \ref{WWKtruthtables} and Figure \ref{KKtruthtables}.

 \begin{figure}[ht]
\centering
{\small
\begin{tabular}{ccc}

\begin{tabular}{c|c}
         & $\neg$ \\ \hline
	$\t$ & $\f$ \\
	$\i$ & $\i$ \\
	$\f$ & $\t$ \\
\end{tabular}

			&
			
\begin{tabular}{c|cccccc}
    $\wedge$ & $\t$ & $\i$ & $\f$ \\ \hline
	$\t$ & $\t$ & $\i$ & $\f$ \\ 
	$\i$ & $\i$ & $\i$ & $\i$ \\  
	$\f$ & $\f$ & $\i$ & $\f$ \\  
\end{tabular}
			
			&
			
\begin{tabular}{c|cccccc}
    $\vee$ & $\t$ & $\i$ & $\f$ \\ \hline
	$\t$ & $\t$ & $\i$ & $\t$ \\ 
	$\i$ & $\i$ & $\i$ & $\i$ \\  
	$\f$ & $\t$ & $\i$ & $\f$ \\  
\end{tabular}
			
\end{tabular}
}
\caption{Weak Kleene truth tables}
\label{WWKtruthtables}
\end{figure}

\begin{figure}[ht]
\centering
{\small
\begin{tabular}{ccc}

\begin{tabular}{c|c}
         & $\neg$ \\ \hline
	$\t$ & $\f$ \\
	$\i$ & $\i$ \\
	$\f$ & $\t$ \\
\end{tabular}

			&
			
\begin{tabular}{c|cccccc}
    $\wedge$ & $\t$ & $\i$ & $\f$ \\ \hline
	$\t$ & $\t$ & $\i$ & $\f$ \\ 
	$\i$ & $\i$ & $\i$ & $\f$ \\  
	$\f$ & $\f$ & $\f$ & $\f$ \\  
\end{tabular}
			
			&
			
\begin{tabular}{c|cccccc}
    $\vee$ & $\t$ & $\i$ & $\f$ \\ \hline
	$\t$ & $\t$ & $\t$ & $\t$ \\ 
	$\i$ & $\t$ & $\i$ & $\i$ \\  
	$\f$ & $\t$ & $\i$ & $\f$ \\  
\end{tabular}
			
\end{tabular}
}
\caption{Strong Kleene truth tables}
\label{KKtruthtables}
\end{figure}

$\mathbb{SS}$ and $\mathbb{TT}$ defined over the Strong Kleene scheme induce the logics $\mathsf{K}_{3}$ and $\mathsf{LP}$, respectively. Their intersection is known as $\mathsf{S}_{3}$ or $\mathsf{KO}$. In the case of the logics defined over the Weak Kleene scheme, $\mathbb{SS}$ corresponds to $\mathsf{WK}$ and $\mathbb{TT}$ to $\mathsf{PWK}$.\footnote{To the best of our knowledge, the intersection of the two has not been studied yet.} Given these logics, Theorem \ref{th:st0} implies the following:
\begin{center}
 $T(\mathbb{K}_{3} \cup \mathbb{PWK})=\mathbb{ST}=\mathbb{CL}$\\
 $T^{d}(\mathbb{K}_{3} \cap \mathbb{PWK})=\emptyset$
\end{center}

 \noindent 
 The choice here of $\mathbb{K}_{3}$ and $\mathbb{PWK}$ was arbitrary. As mentioned earlier, provided that one logic is based on an $\mathbf{ss}$ standard and the other on a $\mathbf{tt}$ standard, the choice of a BNM scheme for each is inconsequential.

A distinctive feature of $\mathsf{K}_{3}$ and $\mathsf{WK}$, and of all the $\bf{ss}$-logics is their paracompleteness: the law of excluded middle $\emptyset\Rightarrow \phi \vee \neg \phi$ fails in each (the valuation that assigns the intermediate value to $\phi$ serves as a counterexample). On the other hand, $\mathsf{LP}$ and $\mathsf{PWK}$ and all the $\bf{tt}$-logics are paraconsistent: the rule of Explosion $\phi \wedge \neg \phi\Rightarrow \psi$ has a counterexample in each (the valuation that assigns the intermediate value to $\phi$ and false to $\psi$ serves as a counterexample). Given our results, classical logic is therefore obtained by operating over the union of a paracomplete and a paraconsistent logic.

\section{Additional remarks on the dual transitive closure operator}
\label{sec:aditional}
As proven in Section \ref{sct:TrDef}, the transitive closure operator and the dual transitive operator are interdefinable via the set-theoretical notion of absolute complement. However, one may ask how this duality manifests when applied to sets of inferences. To address this, we first need to better understand the effects of the dual transitive closure operator. From Section \ref{sct:TrDef}, we know that it is an interior operator, a notion dual to the notion of closure operator, which---unlike its dual---contracts a set rather than extends it. Could there, then, be a property analogous to transitivity that is either imposed or removed by the dual transitive closure operator?

The dual transitive closure of a set of inferences $\mathbb{L}$ is defined as the greatest subset of $\mathbb{L}$ whose complement is closed under transitivity. In light of this definition, it is easy to see that the dual transitive closure operator will not add new inferences to the set. In particular, provided that $\mathsf{L}$ is non-trivial, 
the dual transitive closure operator will remove some instances of reflexivity, turning $\mathbb{L}$ into a non-reflexive relation.\footnote{In this context we take a logic $\mathsf{L}$ to be \textit{trivial} if $\phi \Rightarrow \psi \in \mathbb{L}$ for every $\phi, \psi$.}


\begin{fact}
Let $\mathsf{L}$ be structural, 
\[\begin{array}{ccc}
\text{If } \mathsf{L} \ \text{is non-trivial}, & \text{then} & T^d(\mathbb{L}) \ \text{is non-reflexive}.
\end{array}\]
\end{fact}
\begin{proof}
    Assume that $T^d(\mathbb{L})$ is reflexive. Then for all $\phi \in \mathcal{L}$, $\phi \Rightarrow\phi \in T^d(\mathbb{L})$. So, $p \Rightarrow p \in T^d(\mathbb{L})$ for some arbitrary $p$, and thus, for all $\Delta \neq \emptyset$, $p \Rightarrow \delta \in T^d(\mathbb{L}) \subseteq \mathbb{L}$ for some $\delta \in \Delta$ or $ \Delta \Rightarrow\ p \in T^d(\mathbb{L}) \subseteq \mathbb{L}$. This means in particular that $p \Rightarrow q \in \mathbb{L}$ or $q \Rightarrow p \in \mathbb{L}$ for some arbitrary $q$. In both cases $r \Rightarrow s \in \mathbb{L}$ for two different variables $r$ and $s$, 
    which, by structurality, shows that $\mathsf{L}$ is trivial.


    \end{proof}

We subsequently have two dual operators---the transitive closure operator and the dual transitive operator---turning any non-trivial set of inferences into a transitive one for the first, and into a non-reflexive one for the second. 

Although a relation can only be reflexive in one way, it can fail to be reflexive in numerous ways, depending on which instances of reflexivity are absent. The dual transitive closure only preserves the elements of $\mathbb{L}$ that cannot be obtained by closing its complement under transitivity. As we will see, this set corresponds precisely to the set of \textit{theorems} and \textit{antitheorems} of $\mathsf{L}$.

\begin{Definition}[Theorem and antitheorem]
Given a logic $\mathsf{L}$, a formula $\phi \in \mathcal{L}$ is said to be a \textit{theorem} of $\mathsf{L}$ if $v \models_{\mathsf{L}} \phi$ for all valuations $v$, and an \textit{antitheorem} of $\mathsf{L}$ if $v \not\models_{\mathsf{L}} \phi$ for all valuations $v$.
\end{Definition}

\noindent
By extension, a set of formulas $\Gamma$ is said to be a theorem if for all $v$ and all $\gamma \in \Gamma$, it holds that $v \models_{\mathsf{L}} \gamma$. It is said to be an antitheorem if for all $v$ there is $\gamma \in \Gamma$ such that $v \not\models_{\mathsf{L}} \gamma$.

The notions of theorem and antitheorem can be formally defined in several ways. Some of these definitions are provided in the next fact. We will find these equivalent formulations useful later.

\begin{fact}\label{fct:ath2}
    Let $\mathsf{L}$ be a logic. The following are equivalent:
    \begin{enumerate}[label=(\roman*), align=left]
    \setlength\itemsep{0pt}
        \item\label{Ati2s} $\Gamma$ is an antitheorem of $\mathsf{L}$.
        \item\label{Ati4s} $\models_{\mathsf{L}} \Gamma \Rightarrow \phi$ for all $\phi \in \mathcal{L}$.
        \item\label{Ati5s} $\models_{\mathsf{L}} \Gamma \Rightarrow p$ for some $p \not\in \textup{At}(\Gamma)$.
    \end{enumerate}
\end{fact}

\begin{proof}
    The only non-trivial step is the one from \ref{Ati5s} to \ref{Ati2s}. So, assume that $\Gamma$ is not an antitheorem. It follows that there is a $v$ such that $v \models_{\mathsf{L}} \gamma$ for all $\gamma \in \Gamma$. Let $p \not\in \textup{At}(\Gamma)$ and assume that $v'$ is exactly like $v$ except that if $v\models_{\mathsf{L}} p$, $v'\not\models_{\mathsf{L}} p$. Then $v \models_{\mathsf{L}} \gamma$ for all $\gamma \in \Gamma$ since $p \not\in \textup{At}(\Gamma)$ and hence $\not\models_{\mathsf{L}} \Gamma \Rightarrow p$ for all $p \not\in \textup{At}(\Gamma)$. 
\end{proof}

\begin{fact}\label{fct:th2}
    Let $\mathsf{L}$ be a logic. The following are equivalent:
    \begin{enumerate}[label=(\roman*), align=left]
    \setlength\itemsep{0pt}
        \item\label{Ati2sbis} $\phi$ is a theorem of $\mathsf{L}$.
        \item\label{Ati4sbis} $\models_{\mathsf{L}} \Gamma \Rightarrow \phi$ for all $\Gamma \subseteq \mathcal{L}$.
        \item\label{Ati5sbis} $\models_{\mathsf{L}} p \Rightarrow \phi$ for some $p \not\in \textup{At}(\phi)$.
    \end{enumerate}
\end{fact}

\begin{proof}
    The proof is similar to the proof of Fact \ref{fct:ath2}.
\end{proof}

\noindent
Given these two facts, we will identify the set of theorems of a logic with the set of all valid inferences of the form $\Gamma \Rightarrow \phi$, where $\phi$ is a theorem. Similarly, the set of antitheorems of a logic will be identified with the set of inferences of the form $\Gamma \Rightarrow \phi$, where $\Gamma$ is an antitheorem. For any logic $\mathsf{L}$, we denote the set of its antitheorems and theorems, understood in this broader sense, by $\mathbb{L}^{\star}$. Specifically, we define
\[\mathbb{L}^{\star} \coloneqq \lbrace \Gamma \Rightarrow \phi : \Gamma \Rightarrow \psi \in \mathbb{L} \text{ for all } \psi \ \textup{or} \ \Delta \Rightarrow \phi \in \mathbb{L} \text{ for all } \Delta \rbrace.\]

The next proposition thus expresses that for a given logic $\mathsf{L}$, the operator $T^d$ will select only the theorems and antitheorems of $\mathsf{L}$.

\begin{proposition}\label{prop:atth}
For any logic $\mathsf{L}$,
    \[T^d(\mathbb{L})= \mathbb{L}^{\star}. \]
\end{proposition}

\begin{proof}
    ($\subseteq$) Assume $\Gamma \Rightarrow \phi \not\in \mathbb{L}^{\star}$. Then $\not\models_{\mathsf{L}} \Gamma \Rightarrow p$ and $\not\models_{\mathsf{L}} p \Rightarrow \phi$ for all $p \not\in \textup{At}(\Gamma \cup \lbrace\phi\rbrace)$ by facts \ref{fct:ath2} and \ref{fct:th2}. Hence, there is $p \not\in \textup{At}(\Gamma \cup \lbrace\phi\rbrace)$ such that $\Gamma \Rightarrow p \in \overline{\mathbb{L}}$ and $p \Rightarrow \phi \in \overline{\mathbb{L}}$ since $\textup{At}(\Gamma \cup \lbrace\phi\rbrace)$ finite, and thus $\Gamma \Rightarrow \phi \in T(\overline{\mathbb{L}})$. So, $\Gamma \Rightarrow \phi \not\in \overline{T(\overline{\mathbb{L}})} = T^d(\mathbb{L})$ by Proposition \ref{prp:duaTr}.

    ($\supseteq$) Assume $\Gamma \Rightarrow \phi \in \mathbb{L}^{\star}$. Then $\Gamma$ is an antitheorem or $\phi$ is a theorem. Assume the former. Then $\Gamma\Rightarrow \psi \in \mathbb{L}$ for all $\psi \in \mathcal{L}$. Assume further that $\Gamma \Rightarrow \chi \in T(\overline{\mathbb{L}})$ for some $\chi$ for the sake of contradiction. Consider the set $\mathbb{X} = T(\overline{\mathbb{L}}) - \lbrace \Gamma \Rightarrow \psi : \psi \in \mathcal{L} \rbrace$. Clearly $\overline{\mathbb{L}} \subseteq \mathbb{X} \subset T(\overline{\mathbb{L}})$ since $\Gamma\Rightarrow \psi \not\in \overline{\mathbb{L}}$ for all $\psi$. Now, assume that for all $\Sigma \subseteq \mathcal{L}$ and all $\psi \in \mathcal{L}$, there is $\Delta \neq \emptyset$ such that $\Sigma \Rightarrow \delta \in \mathbb{X}$ for all $\delta \in \Delta$ and $\Delta \Rightarrow \psi \in \mathbb{X}$. The case is impossible for $\Sigma = \Gamma$ or $\Delta = \Gamma$, and otherwise $\Sigma \Rightarrow \delta \in T(\overline{\mathbb{L}})$ for all $\delta \in \Delta$ and $\Delta \Rightarrow \psi \in T(\overline{\mathbb{L}})$. In such a case, $\Sigma \Rightarrow \psi \in T(\overline{\mathbb{L}})$, and $\Sigma \Rightarrow \psi \in \mathbb{X}$ by definition of $\mathbb{X}$, so $\mathbb{X}$ is closed under transitivity. But $T(\overline{\mathbb{L}})$ is the least transitive superset of $\overline{\mathbb{L}}$, so this is impossible. Therefore, $\Gamma \Rightarrow \chi \not\in T(\overline{\mathbb{L}})$ for all $\chi$ and in particular $\Gamma \Rightarrow \phi \not\in T(\overline{\mathbb{L}})$. So, $\Gamma \Rightarrow \phi \in \overline{T(\overline{\mathbb{L}})} = T^d(\mathbb{L})$. The latter case can be proved similarly.
\end{proof}

\noindent
Accordingly, the dual transitive closure will transform any non-trivial set of valid inferences of $\mathsf{L}$ into a non-reflexive one by preserving only the theorems and antitheorems of $\mathsf{L}$.

The results from the previous section can now be reinterpreted in light of this new fact. We showed that the dual transitive closure of the intersection of $\mathbb{SS}$ and $\mathbb{TT}$ coincides with the set of $\mathsf{TS}$-valid inferences. Therefore, the set of $\mathsf{TS}$-valid inference corresponds to the set of theorems and antitheorems of $\mathbb{SS} \cap \mathbb{TT}$, which is known to have none. 

\section{Lattice operations} \label{sec:lattice}

We now further examine the interrelation between the sets of inferences studied throughout this paper, specifically those closed under the transitive closure operator and its dual.

The relationships among the set of valid inferences of the logics $\mathsf{ST}$, $\mathsf{SS}$, $\mathsf{TT}$, and $\mathsf{SS} \cap \mathsf{TT
}$, defined over the Strong Kleene scheme, are typically illustrated using a diamond-shaped Hasse diagram (Figure \ref{fig:semlatp}). In this diagram, $\mathbb{ST}$ (or $\mathbb{CL}$) is positioned at the top, $\mathbb{SS} \cap \mathbb{TT}$ at the bottom, with $\mathbb{SS}$ and $\mathbb{TT}$ in between. While this diagram effectively illustrates the inclusion relation holding between these logics, there is more to consider: the inclusion order actually forms a \textit{lattice}.\footnote{A \textit{lattice} is a partially ordered structure $\textbf{A} = \langle A, \le\rangle$ such that for any pair of elements $x, y \in A$, $x, y$ have both an infimum (join) and a supremum (meet) in $A$.} In the case of the logics defined over the Strong Kleene scheme, this fact is well-known. The lattice of extensions of the Belnap-Dunn four-valued logic $\mathsf{FDE}$---including the logics $\mathsf{CL}$, $\mathsf{K}_3$, $\mathsf{LP}$, and $\mathsf{KO}$---was studied in \cite{Rivieccio2012} and \cite{prenosil_2023}, and shown to form a complete lattice. In this structure, $\mathsf{CL}$ is the least Tarskian logic extending both $\mathsf{K}_3$ and $\mathsf{LP}$, $\mathsf{K}_3 \cap \mathsf{LP}$ is the greatest Tarskian logic extended by both $\mathsf{K}_3$ and $\mathsf{LP}$, while $\mathsf{K}_3$ and $\mathsf{LP}$ are incomparable. Together, these four logics constitute a sublattice of the lattice of extensions of $\mathsf{FDE}$. As we shall prove, 
this result holds not only for logics based on the Strong Kleene scheme but generalizes to all logics defined by the $\bf{ss}$, $\bf{tt}$, $\bf{st}$, and $\bf{ss \cap tt}$ standards over any BNM scheme. For any two of these logics, the least Tarskian logic extended by both correspond to their intersection, while the greatest Tarskian logic extending both is obtained by closing their union under transitivity. Perhaps even more unexpectedly, this generalization still holds when each of the $\bf{ss}$, $\bf{tt}$, $\bf{st}$ logics are defined over a different scheme. 

{\footnotesize 
\begin{figure}[b]
\centering
 \begin{tikzpicture}
\draw (0,4) node (st) {$\mathbb{ST}(= \mathbb{CL})$};
\draw (2,2) node (ss) {$\mathbb{TT}$};
\draw (-2,2) node (tt) {$\mathbb{SS}$};
\draw (0,0) node (ts) {$\mathbb{SS} \cap \mathbb{TT}$};
 \draw[-] (ts) -- (ss);
 \draw[-] (ts) -- (tt);
 \draw[-] (ss) -- (st);
 \draw[-] (tt) -- (st);
\end{tikzpicture}
\caption{Inclusion order over $\mathbb{ST}$, $\mathbb{SS}$, $\mathbb{TT}$ and $\mathbb{SS} \cap \mathbb{TT}$}\label{fig:semlatp}
\end{figure}
}

We begin with the following observation, which establishes that for any two Tarskian logics, the least Tarskian logic extending them corresponds to their transitive closure.

\begin{proposition}
    Let $\mathbb{L}_1 = Tar(\mathbb{L}_1)$ and $\mathbb{L}_2 = Tar(\mathbb{L}_2)$. Then,
    \[ T(\mathbb{L}_1 \cup \mathbb{L}_2) = Tar(\mathbb{L}_1 \cup \mathbb{L}_2).\]
\end{proposition}

\begin{proof}
    If $\mathbb{L}_1$ and $\mathbb{L}_2$ satisfy Reflexivity, then $\mathbb{L}_1 \cup \mathbb{L}_2$ too. By definition, $T(\mathbb{L}_1 \cup \mathbb{L}_2)$ is a superset of $\mathbb{L}_1 \cup \mathbb{L}_2$, so it also satisfies Reflexivity. Turning to Monotonicity, assume that $\Gamma \Rightarrow \phi \in \mathbb{L}_1 \cup \mathbb{L}_2$. Whether $\Gamma \Rightarrow \phi \in \mathbb{L}_1$ or $\Gamma \Rightarrow \phi \in \mathbb{L}_2$, in both cases $\Gamma, \Sigma \Rightarrow \phi  \in \mathbb{L}_1 \cup \mathbb{L}_2 \subseteq T(\mathbb{L}_1 \cup \mathbb{L}_2)$. For Structurality, the case is similar. If $\Gamma \Rightarrow \phi \in \mathbb{L}_1 \cup \mathbb{L}_2$, whether $\Gamma \Rightarrow \phi \in \mathbb{L}_1$ or $\Gamma \Rightarrow \phi \in \mathbb{L}_2$, in both cases $\sigma[\Gamma] \Rightarrow \sigma[\phi]  \in \mathbb{L}_1 \cup \mathbb{L}_2 \subseteq T(\mathbb{L}_1 \cup \mathbb{L}_2)$, for every substitution $\sigma$. Finally, $T(\mathbb{L}_1 \cup \mathbb{L}_2)$ is the smallest transitive superset of $\mathbb{L}_1 \cup \mathbb{L}_2$ by definition and it satisfies Reflexivity, Monotonicity and Structurality, so it is the smallest superset of  $\mathbb{L}_1 \cup \mathbb{L}_2$ satisfying Reflexivity, Monotonicity, Transitivity and Structurality.
\end{proof}

Next, we establish that all the logics discussed thus far, except for $\mathsf{TS}$, are Tarskian.

\begin{lemma}\label{lem:elem}
    For all $\mathbb{L} \in \lbrace \mathbb{ST}, \mathbb{SS}, \mathbb{TT}, \mathbb{SS} \cap \mathbb{TT}\rbrace$, $Tar(\mathbb{L}) = \mathbb{L}$.
\end{lemma}

\begin{proof}
The proofs for Reflexivity, Monotonicity and Structurality are straightforward. We thus focus solely on Transitivity. We prove that if there is $\Delta \neq \emptyset$, $\Gamma \Rightarrow \delta \in \mathbb{L}$ for all $\delta \in \Delta$ and  $\Delta \Rightarrow \phi \in \mathbb{L}$, then $\Gamma \Rightarrow \phi \in \mathbb{L}$. We demonstrate first the case where $\mathbb{L} = \mathbb{SS}$, noting that the case for $\mathbb{TT}$ is similar, while the case for $\mathbb{ST}$ is straightforward. Assume $\Gamma \Rightarrow \phi \not\in \mathbb{SS}$. Then there is $v$ such that $v(\gamma)=1$ for all $\gamma \in \Gamma$ and $v(\phi)\neq 1$. Let $\Delta \neq \emptyset$. If for all $\delta \in \Delta$, $v(\delta) = 1$, then $v \not\models^{\textbf{X}}_{\bf{ss}} \Delta \Rightarrow \phi$. If for some $\delta \in \Delta$, $v(\delta) \neq 1$, then $v \not\models^{\textbf{X}}_{\bf{ss}} \Gamma \Rightarrow \delta$ for some $\delta \in \Delta$. In both cases there is no $\Delta \neq \emptyset$ such that $\Gamma \Rightarrow \delta \in \mathbb{SS}$ for all $\delta \in \Delta$ and  $\Delta \Rightarrow \phi \in \mathbb{SS}$.
Turning to $\mathbb{SS} \cap \mathbb{TT}$, if there is $\Delta$ such that $\Gamma \Rightarrow \delta \in \mathbb{SS} \cap \mathbb{TT}$ for all $\delta \in \Delta$ and  $\Delta \Rightarrow \phi \in \mathbb{SS} \cap \mathbb{TT}$, then $\Gamma \Rightarrow \delta$ is both in $\mathbb{SS}$ and $\mathbb{TT}$ for all $\delta \in \Delta$, and similarly for $\Delta \Rightarrow \phi$. But we just proved that $\mathbb{SS} = Tar(\mathbb{SS})$ and $\mathbb{TT} = Tar(\mathbb{TT})$, so $\Gamma \Rightarrow \phi \in \mathbb{SS} \cap \mathbb{TT}$.
\end{proof}

\noindent
The join of two elements of the lattice will be defined as the Tarskian closure of their union, while their meet will be defined as their intersection. Given the identity $\mathbb{SS} \cap \mathbb{TT} = Tar(\mathbb{SS} \cap \mathbb{TT})$, which we have just proved, we could equally use the Tarskian closure of the intersection as meet, emphasizing the duality of the operations. 

We now show that these operations are well-suited to define the inclusion order, as expected in a lattice. The suitability of the intersection for defining the inclusion order is straightforward, so we will focus on the transitive closure of the union.

\enlargethispage{15pt}

\begin{lemma}\label{lem:elem2}
    For all $\mathbb{X}, \mathbb{Y} \in \lbrace \mathbb{ST}, \mathbb{SS}, \mathbb{TT}, \mathbb{SS} \cap \mathbb{TT}\rbrace$, $\mathbb{X} \subseteq \mathbb{Y}$ if and only if $T(\mathbb{X} \cup \mathbb{Y}) = \mathbb{Y}$.
\end{lemma}

\begin{proof}
    (Left-to-right) Assume  $\mathbb{X} \subseteq \mathbb{Y}$. By Lemma \ref{lem:elem}, $ T(\mathbb{Y}) = \mathbb{Y}$, so $T(\mathbb{X} \cup \mathbb{Y}) = \mathbb{Y}$.
    
\medskip

    \noindent
    (Right-to-left) Assume $T(\mathbb{X} \cup \mathbb{Y}) = \mathbb{Y}$ and let $\Gamma \Rightarrow \phi \in \mathbb{X}$. Then $\Gamma \Rightarrow \phi \in \mathbb{X} \cup \mathbb{Y} \subseteq T(\mathbb{X} \cup \mathbb{Y}) = \mathbb{Y}$.
\end{proof}

We are now ready to prove that the sets of valid inferences of the aforementioned logics form a lattice. First, we show that, together with intersection and the binary operation of transitive closure of the union, it constitutes an algebra. Then, we show that it is a lattice.

\begin{fact}\label{fct:lat}
    Let $\mathfrak{L} \coloneqq \langle \lbrace \mathbb{ST}, \mathbb{SS}, \mathbb{TT}, \mathbb{SS} \cap \mathbb{TT} \rbrace, \sqcap, \sqcup \rangle$, with $\sqcup$ defined for all $\mathbb{X}, \mathbb{Y} \in \mathfrak{L}$ by $\mathbb{X} \sqcup \mathbb{Y} \coloneqq T(\mathbb{X} \cup \mathbb{Y})$ and $\sqcap$ by $\mathbb{X} \sqcap \mathbb{Y} \coloneqq \mathbb{X} \cap \mathbb{Y}$. Then $\mathfrak{L}$ is a lattice algebra.
\end{fact}

\begin{proof}
    To show that $\mathfrak{L}$ is an algebra, it is enough to show that it is closed under $\sqcap$ and $\sqcup$. The closure under $\sqcap$ is straightforward to prove since $\sqcap = \cap$. The closure under $\sqcup$ follows from Theorem \ref{th:st0}, and lemmas \ref{lem:elem} and \ref{lem:elem2}.

    In order to prove that $\mathfrak{L}$ is in addition a lattice algebra, we must show that $\sqcap$ and $\sqcup$ are commutative, idempotent, associative and satisfy the absorbtion laws. The first three properties obviously hold for $\sqcap$, since $\sqcap = \cap$.
    As for $\sqcup$, the first two properties are trivial to establish given the commutativity and the idempotence of union. We focus on proving associativity.

Since $\X \cup \Y \subseteq \X \cup \Y \cup \Z$ and $T$ is a closure operator, $T(\X \cup \Y) \subseteq T(\X \cup \Y \cup \Z)$. But $\Z \subseteq T(\X \cup \Y \cup \Z)$, so $T(\X \cup \Y) \cup \Z \subseteq T(\X \cup \Y \cup \Z)$ and hence $T(T(\X \cup \Y) \cup \Z)) \subseteq T(T(\X \cup \Y \cup \Z)) = T(\X \cup \Y \cup \Z)$. On the other hand, $\Y \cup \Z \subseteq T(\Y \cup \Z)$, so $\X \cup \Y \cup \Z \subseteq \X \cup T(\Y \cup \Z)$ and therefore $T(\X \cup \Y \cup \Z) \subseteq T(\X \cup T(\Y \cup \Z))$. Whence $(\X \sqcup \Y) \sqcup \Z = T(T(\X \cup \Y) \cup \Z) \subseteq T(\X \cup T(\Y \cup \Z)) = \X \sqcup (\Y \sqcup \Z)$. The right-to-left inclusion is similar.

Turning to the absorption laws, since $\mathbb{X} = T(\mathbb{X})$ by Lemma \ref{lem:elem}, it holds that $\mathbb{X} = T(\mathbb{X} \cup (\mathbb{X} \cap \mathbb
{Y}))$, and thus $\mathbb{X} = \mathbb{X} \sqcup (\mathbb{X} \sqcap \mathbb
{Y})$. Now, $\mathbb{X} \subseteq T(\mathbb{X} \cup \mathbb{Y})$, so $\mathbb{X} = \mathbb{X} \cap T(\mathbb{X} \cup \mathbb{Y})$, meaning that $\mathbb{X} = \mathbb{X} \sqcap (\mathbb{X} \sqcup \mathbb{Y})$.
\end{proof}

It is now easy to see that $\mathfrak{L} \coloneqq \langle \lbrace \mathbb{ST}, \mathbb{SS}, \mathbb{TT}, \mathbb{SS} \cap \mathbb{TT}\rbrace, \subseteq \rangle$ is a lattice ordered by inclusion. 

\begin{corollary}
    $\mathfrak{L} \coloneqq \langle \lbrace \mathbb{ST}, \mathbb{SS}, \mathbb{TT}, \mathbb{SS} \cap \mathbb{TT}\rbrace, \subseteq \rangle$ is a lattice with infimum $\textit{inf}\lbrace\X, \Y\rbrace = \X \cap \Y$ and supremum $\textit{sup}\lbrace\X, \Y\rbrace = T(\X \cup \Y)$.
\end{corollary}
\begin{proof}
    The proof is straightforward, given the interdefinability of a lattice order and a lattice algebra.
\end{proof}

Importantly, 
$\mathfrak{L}$ forms a sublattice of the closure system $\mathfrak{T} = \lbrace \mathbb{L} : \mathbb{L} = Tar(\mathbb{L}) \rbrace$ ordered by $\subseteq$. As a closure system, $\langle \mathfrak{T}, \subseteq \rangle$ is a complete lattice of sets of inferences closed under the Tarskian properties.\footnote{A closure system $\mathfrak{C}$ on a set $A$ is a collection of subsets of $\mathcal{P}(A)$ such that $A \in \mathfrak{C}$, and if $\mathfrak{B} \subseteq \mathfrak{C}$ for $\mathfrak{B} \neq \emptyset$, then $\bigcap \mathfrak{B} \in \mathfrak{C}$. Every closure operator $C$ naturally generates a closure system $\mathfrak{C} = \lbrace X \subseteq A : C(X) = X \rbrace$ and a partially ordered set $\langle \mathfrak{C}, \subseteq \rangle$. This partially ordered set can be proved to be a complete lattice with operations $\bigwedge F_i = \bigcap F_i$ and $\bigvee F_i = C(\bigcup F_i)$ for any $\lbrace F_i : i \in I \rbrace \subseteq \mathfrak{C}$. See \cite[p.\ 37]{font2016}, Proposition 1.28.}  For any two elements of $\mathfrak{T}$, their join is the least set of inferences satisfying the Tarskian properties extending both, and their meet is the greatest such set extended by both. Consequently, the fact that $\mathfrak{L}$ is a sublattice of $\mathfrak{T}$ implies that $\mathbb{ST}$ (or $\mathbb{CL}$) is the least set closed under the Tarskian properties extending both $\mathbb{SS}$ and $\mathbb{TT}$, while $\mathbb{SS} \cap \mathbb{TT}$ is the greatest such set contained in both. Given our definition of a logic as the family of sets of its valid inferences, it follows directly that $\mathsf{ST}$ is the least Tarskian logic that extends both $\mathsf{SS}$ and $\mathsf{TT}$, while $\mathsf{SS} \cap \mathsf{TT}$ is the greatest Tarskian logic contained in both. More interestingly, this result holds independently of a choice of BNM scheme for $\mathsf{SS}$, $\mathsf{TT}$ and $\mathsf{ST}$, as illustrated by Figure \ref{fig:pwkk3} for the valid inferences of the logics $\mathsf{CL}$, $\mathsf{K}_3$, $\mathsf{PWK}$, and $\mathsf{K}_3 \cap \mathsf{PWK}$.

{\footnotesize 
\begin{figure}[ht]
\centering
 \begin{tikzpicture}
\draw (0,4) node (st) {$\mathbb{CL}$};
\draw (2,2) node (ss) {$\mathbb{PWK}$};
\draw (-2,2) node (tt) {$\mathbb{K}_3$};
\draw (0,0) node (ts) {$\mathbb{K}_3 \cap \mathbb{PWK}$};
 \draw[-] (ts) -- (ss);
 \draw[-] (ts) -- (tt);
 \draw[-] (ss) -- (st);
 \draw[-] (tt) -- (st);
\end{tikzpicture}
\caption{Lattice order over $\mathbb{CL}$, $\mathbb{K}_3$, $\mathbb{PWK}$ and $\mathbb{K}_3 \cap \mathbb{PWK}$}\label{fig:pwkk3}
\end{figure}
}

Given the duality between the operations of transitive closure and dual transitive closure, what should we expect from the structure that arises from the dual transitive closure of the logics discussed here? It turns out that the resulting structure forms a lattice isomorphic to the one generated by the logics $\mathsf{SS}$, $\mathsf{TT}$, $\mathsf{SS} \cap \mathsf{TT}$, and $\mathsf{ST}$. Specifically, in view of Proposition \ref{prop:atth}, the elements of this lattice are the sets of theorems and antitheorems of these logics. Similar to the lattice of valid inferences, the overall structure holds independently of a specific choice of scheme for $\mathsf{SS}$ and  $\mathsf{TT}$ and $\mathsf{TS}$. 

To prove our claim, we first establish how the different sets are ordered by inclusion.

\begin{fact}\label{fct:subsath}
    Let $\mathsf{SS}, \mathsf{TT}, \mathsf{TS}$ be logics defined respectively by an $\bf{ss}, \bf{tt}$ and $\bf{ts}$ standard over a (possibly different) BNM scheme. Then
    \begin{itemize}
        \item $\mathbb{TS}^{\star} \subseteq \mathbb{SS}^{\star}, \mathbb{TT}^{\star}$,
        \item $\mathbb{SS}^{\star}$ and $\mathbb{TT}^{\star}$ are incomparable.
    \end{itemize}
\end{fact}
\begin{proof}
    For the first inclusion, note that $\mathbb{TS} = \emptyset \subseteq \mathbb{SS}, \mathbb{TT}$. It then follows that $T^d(\mathbb{TS}) \subseteq T^d(\mathbb{SS}), T^d(\mathbb{TT})$ since $T^d$ is an interior operator. But $T^d(\mathbb{TS}) = \mathbb{TS}^{\star}$, $T^d(\mathbb{SS}) = \mathbb{SS}^{\star}$, and $T^d(\mathbb{TT}) = \mathbb{TT}^{\star}$ by Proposition \ref{prop:atth}. Hence, $\mathbb{TS}^{\star} \subseteq \mathbb{SS}^{\star}, \mathbb{TT}^{\star}$.
     
     As for the incomparability of $\mathbb{SS}^{\star}$ and $\mathbb{TT}^{\star}$, note that $p \wedge \neg p \Rightarrow q$ is in $\mathbb{SS}^{\star}$, but not in $\mathbb{TT}^{\star}$, while $p \Rightarrow q \lor \neg q$ is in $\mathbb{TT}^{\star}$, but not in $\mathbb{SS}^{\star}$.
\end{proof}

We know from Proposition \ref{prop:atth} that each of $\mathbb{SS}^{\star}$, $\mathbb{TT}^{\star}$ and $\mathbb{TS}^{\star}$ are open under dual transitivity. We show here that this is also the case of $\mathbb{SS}^{\star} \cup \mathbb{TT}^{\star}$, which corresponds to the set of antitheorems and theorems of classical logic.\footnote{For $\mathsf{SS} = \mathsf{K}_3$ and $\mathsf{TT} = \mathsf{LP}$, this result was previously established in \cite{blomet2024sttsproductsum}.}

\begin{fact}\label{fct:atthcl}
    \[T^d(\mathbb{SS}^{\star} \cup \mathbb{TT}^{\star}) =  \mathbb{SS}^{\star} \cup \mathbb{TT}^{\star} = \mathbb{ST}^{\star} = \mathbb{CL}^{\star}.\]
\end{fact}

\begin{proof}
    The left-to-right inclusion of the first identity is straightforward since $T^d$ is an interior operator. For the right-to-left direction, note that $T^d(\mathbb{SS}^{\star}) = T^d(T^d(\mathbb{SS})) = T^d(\mathbb{SS}) = \mathbb{SS}^{\star} \subseteq T^d(\mathbb{SS}^{\star} \cup \mathbb{TT}^{\star})$ and $T^d(\mathbb{TT}^{\star}) = T^d(T^d(\mathbb{TT})) = T^d(\mathbb{TT}) = \mathbb{TT}^{\star} \subseteq T^d(\mathbb{SS}^{\star} \cup \mathbb{TT}^{\star})$ given Proposition \ref{prop:atth} and the fact that $T^d$ is an interior operator. So $\mathbb{SS}^{\star} \cup \mathbb{TT}^{\star} \subseteq T^d(\mathbb{SS}^{\star} \cup \mathbb{TT}^{\star})$. Turning to the second identity, for the left-to-right inclusion, we know from Theorem \ref{th:ssunionttinclcl} that $\mathbb{SS}, \mathbb{TT} \subseteq \mathbb{SS} \cup \mathbb{TT} \subset \mathbb{ST}$. Therefore $T^d(\mathbb{SS}), T^d(\mathbb{TT}) \subseteq T^d(\mathbb{ST})$ and hence $\mathbb{SS}^{\star} \cup \mathbb{TT}^{\star}  \subseteq \mathbb{ST}^{\star}$. For the right-to-left direction, Assume that $\Gamma \Rightarrow \phi \in \mathbb{ST}^{\star}$. If $\Gamma$ is an antitheorem, then for all $\mathbf{X}$-valuations $v$, $v(\gamma)\neq 1$ for some $\gamma \in \Gamma$, so $\Gamma \Rightarrow \phi \in \mathbb{SS}^{\star}$. If $\phi$ is a theorem, then for all $\mathbf{X}$-valuations $v$, $v(\phi)\in\{1,\sfrac{1}{2}\}$, so $\Gamma \Rightarrow \phi \in \mathbb{TT}^{\star}$. Thus, in both cases, $\Gamma \Rightarrow \phi \in \mathbb{SS}^{\star} \cup \mathbb{TT}^{\star}$. The last identity is a direct consequence of Theorem \ref{thm:cl=st}.
\end{proof}

We then prove a lemma demonstrating the interdefinability of the inclusion order and the dual transitive closure of intersection, following a similar approach to the proof of Lemma \ref{lem:elem2}. The interdefinability of the inclusion order and union is straightforward to prove and is omitted.

\begin{lemma}\label{lem:elemduan}
    For any $\mathbb{X}^{\star}, \mathbb{Y}^{\star} \in \lbrace \mathbb{SS}^{\star}, \mathbb{TT}^{\star}, \mathbb{TS}^{\star}, \mathbb{SS}^{\star} \cup \mathbb{TT}^{\star}\rbrace$, $\mathbb{X}^{\star} \subseteq \mathbb{Y}^{\star}$ if and only if $T^d(\mathbb{X}^{\star} \cap \mathbb{Y}^{\star}) = \mathbb{X}^{\star}$.
\end{lemma}

\begin{proof}
    (Left-to-right) Assume  $\mathbb{X}^{\star} \subseteq \mathbb{Y}^{\star}$. By Proposition \ref{prop:atth} and Fact \ref{fct:atthcl}, $\mathbb{X}^{\star} = T^d(\mathbb{X}) = T^d(T^d(\mathbb{X})) = T^d(\mathbb{X}^{\star})$, so $T^d(\mathbb{X}^{\star} \cap \mathbb{Y}^{\star}) = \mathbb{X}^{\star}$.

    (Right-to-left) Assume $T^d(\mathbb{X}^{\star} \cap \mathbb{Y}^{\star}) = \mathbb{X}^{\star}$ and let $\Gamma \Rightarrow \phi \in \mathbb{X}^{\star}$. Then $\Gamma \Rightarrow \phi \in T^d(\mathbb{X}^{\star} \cap \mathbb{Y}^{\star}) \subseteq \mathbb{X}^{\star} \cap \mathbb{Y}^{\star} \subseteq \mathbb{Y}^{\star}$ since $T^d$ is an interior operator.
\end{proof}

With this lemma at hand, we can demonstrate that $\langle \lbrace \mathbb{SS}^{\star}, \mathbb{TT}^{\star}, \mathbb{TS}^{\star}, \mathbb{SS}^{\star} \cup \mathbb{TT}^{\star}\rbrace, \sqcap, \sqcup \rangle$ is a lattice algebra with meet $\sqcap$ defined by $\mathbb{X}^{\star} \sqcap \mathbb{Y}^{\star} \coloneqq T^d(\mathbb{X}^{\star} \cap \mathbb{Y}^{\star})$ and join $\sqcup$ by $\mathbb{X}^{\star} \cup \mathbb{Y}^{\star}$.

\begin{fact}
    Let $\mathfrak{L}^{\star} \coloneqq \langle \lbrace \mathbb{SS}^{\star}, \mathbb{TT}^{\star}, \mathbb{TS}^{\star}, \mathbb{SS}^{\star} \cup \mathbb{TT}^{\star}\rbrace, \sqcap, \sqcup \rangle$, with $\sqcap$ defined for all $\mathbb{X}^{\star}, \mathbb{Y}^{\star} \in \mathfrak{L}^{\star}$ by $\mathbb{X}^{\star} \sqcap \mathbb{Y}^{\star} \coloneqq T^d(\mathbb{X}^{\star} \cap \mathbb{Y}^{\star})$ and $\sqcup$ by $\mathbb{X}^{\star} \sqcup \mathbb{Y}^{\star} \coloneqq \mathbb{X}^{\star} \cup \mathbb{Y}^{\star}$. Then $\mathfrak{L}^{\star}$ is a lattice algebra.
\end{fact}

\begin{proof}
    It follows from Proposition \ref{prop:atth} that all the elements of $\mathfrak{L}^{\star}$ are closed under $T^d$. By Theorem \ref{thm:tssum}, $T^d(\mathbb{SS} \cap \mathbb{TT}) = \mathbb{TS}$, so $T^d(T^d(\mathbb{SS} \cap \mathbb{TT})) = T^d(\mathbb{TS})$. But $T^d(T^d(\mathbb{SS} \cap \mathbb{TT})) = T^d(\mathbb{SS} \cap \mathbb{TT})$ since $T^d$ is an interior operator and $T^d(\mathbb{TS}) = \mathbb{TS}^{\star}$ by Proposition \ref{prop:atth}. Thus, $T^d(\mathbb{SS}^{\star} \cap \mathbb{TT}^{\star}) = \mathbb{TS}^{\star}$. Any other pair is comparable by Fact \ref{fct:subsath}, so the dual transitive closure of their intersection corresponds to the least of the two. Hence, the carrier set of $\mathfrak{L}^{\star}$ is closed under $\sqcap$. Evidently, the carrier set of $\mathfrak{L}^{\star}$ is also closed under $\sqcup$ since $\sqcup = \cup$. $\mathfrak{L}^{\star}$ is therefore an algebra. The remainder of the proof is dual to Fact \ref{fct:lat}.
\end{proof}

\begin{corollary}
   $\mathfrak{L}^{\star} \coloneqq \langle \lbrace \mathbb{TS}^{\star}, \mathbb{SS}^{\star}, \mathbb{TT}^{\star}, \mathbb{SS}^{\star} \cup \mathbb{TT}^{\star}\rbrace, \subseteq \rangle$ is a lattice with infimum $\textit{inf}\lbrace\X^{\star}, \Y^{\star}\rbrace = T^d(\X^{\star} \cap \Y^{\star})$ and supremum $\textit{sup}\lbrace\X^{\star}, \Y^{\star}\rbrace = \X^{\star} \cup \Y^{\star}$.
\end{corollary}

\begin{proof}
    Straightforward.
\end{proof}

{\footnotesize 
\begin{figure}[ht]
\centering
 \begin{tikzpicture}
\draw (0,4) node (st) {$\mathbb{SS}^{\star} \cup \mathbb{TT}^{\star}$};
\draw (2,2) node (ss) {$\mathbb{TT}^{\star}$};
\draw (-2,2) node (tt) {$\mathbb{SS}^{\star}$};
\draw (0,0) node (ts) {$\mathbb{TS}^{\star}$};
 \draw[-] (ts) -- (ss);
 \draw[-] (ts) -- (tt);
 \draw[-] (ss) -- (st);
 \draw[-] (tt) -- (st);
\end{tikzpicture}
\caption{Inclusion order over $\mathbb{SS}^{\star} \cup \mathbb{TT}^{\star}, \mathbb{SS}^{\star}, \mathbb{TT}^{\star}$ and $\mathbb{TS}^{\star}$}\label{fig:thath}
\end{figure}
}

Dually to $\mathfrak{L}$, the structure $\mathfrak{L}^{\star}$ is a sublattice of the interior system $\mathfrak{T}^{d} = \lbrace \mathbb{L} : \mathbb{L} = T^d(\mathbb{L}) \rbrace$ ordered by $\subseteq$. $\langle \mathfrak{T}^d, \subseteq \rangle$ is a complete lattice of sets of inferences open under dual transitivity.\footnote{An interior system $\mathfrak{I}$ on a set $A$ is a collection of subsets of $\mathcal{P}(A)$ such that $\emptyset \in \mathfrak{I}$, and if $\mathfrak{B} \subseteq \mathfrak{C}$ for $\mathfrak{B} \neq \emptyset$, then $\bigcup \mathfrak{B} \in \mathfrak{I}$. Every interior operator $I$ naturally generates an interior system $\mathfrak{I} = \lbrace X \subseteq A : I(X) = X \rbrace$ and a partially ordered set $\langle \mathfrak{I}, \subseteq \rangle$. This partially ordered set can be proved to be a complete lattice with operations $\bigvee F_i = \bigcup F_i$ and $\bigwedge F_i = I(\bigcap F_i)$ for any $\lbrace F_i : i \in I \rbrace \subseteq \mathfrak{I}$. See \cite[p.\ 42]{font2016}.} By Proposition \ref{prop:atth}, $\mathfrak{T}^{d}$ can alternatively be viewed as a complete lattice of sets of theorems and antitheorems. For any two elements of $\mathfrak{T}$, their join is the least set of theorems and antitheorems extending both, and their meet is the greatest set of theorems and antitheorems extended by both. In consequence, the fact that $\mathfrak{L}^{\star}$ is a sublattice of $\mathfrak{T}^d$ implies that $\mathbb{CL}^{\star}$ is the least set of theorems and antitheorems extending both $\mathbb{SS}^{\star}$ and $\mathbb{TT}^{\star}$, and $\mathbb{SS} \cap \mathbb{TT}$ the greatest  set of theorems and antitheorems extended by both.

Figure \ref{fig:thath} represents the lattice of theorems and antitheorems of the $\mathsf{SS}, \mathsf{TT}, \mathsf{ST}$ and $\mathsf{TS}$ logics. It is easy to see that it is isomorphic to the lattice represented in Figure \ref{fig:semlatp}. By Fact \ref{fct:atthcl}, the top of the lattice, $\mathbb{SS}^{\star} \cup \mathbb{TT}^{\star}$, corresponds to the set of antitheorems and theorems of classical logic. Conversely, the bottom of the lattice, $\mathbb{TS}^{\star}$, aligns with the set of antitheorems and theorems of $\mathsf{SS} \cap \mathsf{TT}$. The duality of the transitive closure operator and the dual transitive closure operator reverberates on the meet and join of the respective lattices. In the lattice of valid inferences, as we have seen, $\mathbb{SS} \cup \mathbb{TT}$ does not coincide with $\mathbb{ST}$, the set of classically valid inferences. Dually, in the lattice of antitheorems and theorems, $\mathbb{SS}^{\star} \cap \mathbb{TT}^{\star}$ does not coincide with $\mathbb{TS}$, the set of antitheorems and theorems of $\mathsf{SS} \cap \mathsf{TT}$, for $(p \wedge \neg p) \Rightarrow (q \lor \neg q) \in \mathbb{SS}^{\star} \cap \mathbb{TT}^{\star}$. In both cases, specific operations must be applied to obtain the desired set: for the union, the transitive closure operator is used; for the intersection, the dual transitive closure operator is applied. 



\section{Conclusion}\label{sect:conclusion}

In this paper, we presented several results concerning the relation between $\bf{tt}$- and $\bf{ss}$- logics and their corresponding $\bf{st}$- and $\bf{ts}$-logics over three-valued BNM schemes. 

On the one hand, we showed that $\mathbb{SS} \cup \mathbb{TT} \subsetneq \mathbb{ST}=\mathbb{CL}$, for any BNM scheme. We also introduced the transitive closure operator $T$ and we proved that $T(\mathbb{SS} \cup \mathbb{TT}) = \mathbb{ST}$ (see theorems \ref{th:ssunionttinclcl}, \ref{th:st0}). On the other hand, we proved that $\mathbb{SS} \cap \mathbb{TT} \varsupsetneq \mathbb{TS}$, for any BNM scheme. We introduced the dual transitive closure operator $T^d$, and proved that $T^{d}(\mathbb{SS} \cap \mathbb{TT}) = \mathbb{TS}$ (see Theorem  \ref{thm:tssum}). 

We generalized and extended results from \cite{da2023three}, highlighting notable relationships between $\bf{ss}$- and $\bf{tt}$-logics, concerning properties such as paraconsistency and paracompleteness.

We concluded with observations on the abstract relationships between the transitive closure and its dual. It was noted that, just as the operator $T$ closes a set of inferences under transitivity, its dual $T^d$ eliminates reflexivity in all non-trivial cases, preserving only theorems and antitheorems. Building on this, we demonstrated how the sets $\mathbb{SS}$, $\mathbb{TT}$, $\mathbb{CL}$, and $\mathbb{SS} \cap \mathbb{TT}$ form a lattice, with $\mathbb{CL}$ at the top, extending both $\mathbb{SS}$ and $\mathbb{TT}$, and $\mathbb{SS} \cap \mathbb{TT}$ at the bottom, extended by both. We concluded the previous section by observing that the respective sets of theorems and antitheorems of each of these logics notably followed the same lattice structure.

However, several open questions remain. All the results in this paper are specific to BNM three-valued schemes. As shown by \cite{da2023three}, other types of schemes---such as the Boolean Normal Truth Collapsible and Falsity Collapsible ones---are also capable of yielding classical logic under an $\mathbf{st}$ standard. Yet, logics defined over these schemes exhibit very different properties from those studied here. For example, logics over Boolean Normal Collapsible schemes with a $\mathbf{tt}$ or $\mathbf{ss}$ standard are coextensive with classical logic, and hence they are neither paraconsistent nor paracomplete. Furthermore, $\mathbf{ts}$-logics over these schemes are not necessarily empty. Additionally, the generalization of the results in this article to multiple conclusions and metainferential logics remains to be addressed. We are working on these questions and will address them in a subsequent article.

\bibliographystyle{apacite}
\bibliography{references}

\begin{thebibliography}{}

\bibitem [\protect \citeauthoryear {%
Asenjo%
}{%
Asenjo%
}{%
{\protect \APACyear {1966}}%
}]{%
asenjo1966}
\APACinsertmetastar {%
asenjo1966}%
\begin{APACrefauthors}%
Asenjo, F\BPBI G.%
\end{APACrefauthors}%
\unskip\
\newblock
\APACrefYearMonthDay{1966}{}{}.
\newblock
{\BBOQ}\APACrefatitle {A calculus of antinomies} {A calculus of antinomies}.{\BBCQ}
\newblock
\APACjournalVolNumPages{Notre Dame Journal of Formal Logic}{15}{}{497--509}.
\PrintBackRefs{\CurrentBib}

\bibitem [\protect \citeauthoryear {%
Beaver%
\ \BBA {} Krahmer%
}{%
Beaver%
\ \BBA {} Krahmer%
}{%
{\protect \APACyear {2001}}%
}]{%
BeaverKrahmer2001}
\APACinsertmetastar {%
BeaverKrahmer2001}%
\begin{APACrefauthors}%
Beaver, D.%
\BCBT {}\ \BBA {} Krahmer, E.%
\end{APACrefauthors}%
\unskip\
\newblock
\APACrefYearMonthDay{2001}{}{}.
\newblock
{\BBOQ}\APACrefatitle {A partial account of presupposition projection} {A partial account of presupposition projection}.{\BBCQ}
\newblock
\APACjournalVolNumPages{Journal of Logic, Language and Information}{10}{2}{147--182}.
\PrintBackRefs{\CurrentBib}

\bibitem [\protect \citeauthoryear {%
Blomet%
\ \BBA {} \'Egr\'e%
}{%
Blomet%
\ \BBA {} \'Egr\'e%
}{%
{\protect \APACyear {2024}}%
}]{%
blomet2024sttsproductsum}
\APACinsertmetastar {%
blomet2024sttsproductsum}%
\begin{APACrefauthors}%
Blomet, Q.%
\BCBT {}\ \BBA {} \'Egr\'e, P.%
\end{APACrefauthors}%
\unskip\
\newblock
\APACrefYearMonthDay{2024}{}{}.
\newblock
{\BBOQ}\APACrefatitle {{ST} and {TS} as product and sum} {{ST} and {TS} as product and sum}.{\BBCQ}
\newblock
\APACjournalVolNumPages{Journal of Philosophical Logic}{53}{6}{1673–-1700}.
\PrintBackRefs{\CurrentBib}

\bibitem [\protect \citeauthoryear {%
Bochvar%
}{%
Bochvar%
}{%
{\protect \APACyear {1938}}%
}]{%
Bochvar1938}
\APACinsertmetastar {%
Bochvar1938}%
\begin{APACrefauthors}%
Bochvar, D.%
\end{APACrefauthors}%
\unskip\
\newblock
\APACrefYearMonthDay{1938}{}{}.
\newblock
{\BBOQ}\APACrefatitle {On a three-valued calculus and its application in the analysis of the paradoxes of the extended functional calculus} {On a three-valued calculus and its application in the analysis of the paradoxes of the extended functional calculus}.{\BBCQ}
\newblock
\APACjournalVolNumPages{Matematicheskii Sbornik}{4}{}{287--308}.
\PrintBackRefs{\CurrentBib}

\bibitem [\protect \citeauthoryear {%
Chemla%
, \'Egr{\'e}%
\BCBL {}\ \BBA {} Spector%
}{%
Chemla%
\ \protect \BOthers {.}}{%
{\protect \APACyear {2017}}%
}]{%
Chemlaetal2017}
\APACinsertmetastar {%
Chemlaetal2017}%
\begin{APACrefauthors}%
Chemla, E.%
, \'Egr{\'e}, P.%
\BCBL {}\ \BBA {} Spector, B.%
\end{APACrefauthors}%
\unskip\
\newblock
\APACrefYearMonthDay{2017}{}{}.
\newblock
{\BBOQ}\APACrefatitle {Characterizing logical consequence in many-valued logic} {Characterizing logical consequence in many-valued logic}.{\BBCQ}
\newblock
\APACjournalVolNumPages{Journal of Logic and Computation}{27}{7}{2193-2226}.
\PrintBackRefs{\CurrentBib}

\bibitem [\protect \citeauthoryear {%
Cobreros%
, \'Egr{\'e}%
, Ripley%
\BCBL {}\ \BBA {} van Rooij%
}{%
Cobreros%
\ \protect \BOthers {.}}{%
{\protect \APACyear {2012}}%
}]{%
Cobrerosetal2012}
\APACinsertmetastar {%
Cobrerosetal2012}%
\begin{APACrefauthors}%
Cobreros, P.%
, \'Egr{\'e}, P.%
, Ripley, E.%
\BCBL {}\ \BBA {} van Rooij, R.%
\end{APACrefauthors}%
\unskip\
\newblock
\APACrefYearMonthDay{2012}{}{}.
\newblock
{\BBOQ}\APACrefatitle {Tolerant, classical, strict} {Tolerant, classical, strict}.{\BBCQ}
\newblock
\APACjournalVolNumPages{Journal of Philosophical Logic}{41}{2}{347-385}.
\PrintBackRefs{\CurrentBib}

\bibitem [\protect \citeauthoryear {%
Cobreros%
, \'Egr{\'e}%
, Ripley%
\BCBL {}\ \BBA {} van Rooij%
}{%
Cobreros%
\ \protect \BOthers {.}}{%
{\protect \APACyear {2013}}%
}]{%
Cobrerosetal2013}
\APACinsertmetastar {%
Cobrerosetal2013}%
\begin{APACrefauthors}%
Cobreros, P.%
, \'Egr{\'e}, P.%
, Ripley, E.%
\BCBL {}\ \BBA {} van Rooij, R.%
\end{APACrefauthors}%
\unskip\
\newblock
\APACrefYearMonthDay{2013}{}{}.
\newblock
{\BBOQ}\APACrefatitle {Reaching transparent truth} {Reaching transparent truth}.{\BBCQ}
\newblock
\APACjournalVolNumPages{Mind}{122}{488}{841-866}.
\PrintBackRefs{\CurrentBib}

\bibitem [\protect \citeauthoryear {%
{Da R\'e}%
, Szmuc%
, Chemla%
\BCBL {}\ \BBA {} \'{E}gr\'e%
}{%
{Da R\'e}%
\ \protect \BOthers {.}}{%
{\protect \APACyear {2023}}%
}]{%
da2023three}
\APACinsertmetastar {%
da2023three}%
\begin{APACrefauthors}%
{Da R\'e}, B.%
, Szmuc, D.%
, Chemla, E.%
\BCBL {}\ \BBA {} \'{E}gr\'e, P.%
\end{APACrefauthors}%
\unskip\
\newblock
\APACrefYearMonthDay{2023}{}{}.
\newblock
{\BBOQ}\APACrefatitle {On three-valued presentations of classical logic} {On three-valued presentations of classical logic}.{\BBCQ}
\newblock
\APACjournalVolNumPages{The Review of Symbolic Logic}{}{}{1--23}.
\PrintBackRefs{\CurrentBib}

\bibitem [\protect \citeauthoryear {%
Ferguson%
}{%
Ferguson%
}{%
{\protect \APACyear {2023}}%
}]{%
ferguson2023monstrous}
\APACinsertmetastar {%
ferguson2023monstrous}%
\begin{APACrefauthors}%
Ferguson, T\BPBI M.%
\end{APACrefauthors}%
\unskip\
\newblock
\APACrefYearMonthDay{2023}{}{}.
\newblock
{\BBOQ}\APACrefatitle {Monstrous content and the bounds of discourse} {Monstrous content and the bounds of discourse}.{\BBCQ}
\newblock
\APACjournalVolNumPages{Journal of Philosophical Logic}{52}{1}{111--143}.
\PrintBackRefs{\CurrentBib}

\bibitem [\protect \citeauthoryear {%
Font%
}{%
Font%
}{%
{\protect \APACyear {2016}}%
}]{%
font2016}
\APACinsertmetastar {%
font2016}%
\begin{APACrefauthors}%
Font, J\BPBI M.%
\end{APACrefauthors}%
\unskip\
\newblock
\APACrefYear{2016}.
\newblock
\APACrefbtitle {Abstract Algebraic Logic: An Introductory Textbook} {Abstract algebraic logic: An introductory textbook}\ (\BVOL~60).
\newblock
\APACaddressPublisher{}{College Publications}.
\PrintBackRefs{\CurrentBib}

\bibitem [\protect \citeauthoryear {%
George%
}{%
George%
}{%
{\protect \APACyear {2014}}%
}]{%
George2014}
\APACinsertmetastar {%
George2014}%
\begin{APACrefauthors}%
George, B\BPBI R.%
\end{APACrefauthors}%
\unskip\
\newblock
\APACrefYearMonthDay{2014}{}{}.
\newblock
{\BBOQ}\APACrefatitle {Some remarks on certain trivalent accounts of presupposition projection} {Some remarks on certain trivalent accounts of presupposition projection}.{\BBCQ}
\newblock
\APACjournalVolNumPages{Journal of Applied Non-classical Logics}{24}{1-2}{86--117}.
\PrintBackRefs{\CurrentBib}

\bibitem [\protect \citeauthoryear {%
Halld{\'e}n%
}{%
Halld{\'e}n%
}{%
{\protect \APACyear {1949}}%
}]{%
Hallden1949}
\APACinsertmetastar {%
Hallden1949}%
\begin{APACrefauthors}%
Halld{\'e}n, S.%
\end{APACrefauthors}%
\unskip\
\newblock
\APACrefYear{1949}.
\unskip\
\newblock
\APACrefbtitle {The Logic of Nonsense} {The logic of nonsense}\ \APACtypeAddressSchool {\BUPhD}{}{}.
\unskip\
\newblock
\APACaddressSchool {}{Uppsala University}.
\PrintBackRefs{\CurrentBib}

\bibitem [\protect \citeauthoryear {%
Kleene%
}{%
Kleene%
}{%
{\protect \APACyear {1952}}%
}]{%
Kleene1952-KLEITM}
\APACinsertmetastar {%
Kleene1952-KLEITM}%
\begin{APACrefauthors}%
Kleene, S\BPBI C.%
\end{APACrefauthors}%
\unskip\
\newblock
\APACrefYear{1952}.
\newblock
\APACrefbtitle {Introduction to Metamathematics} {Introduction to metamathematics}.
\newblock
\APACaddressPublisher{Groningen}{P. Noordhoff N.V.}
\PrintBackRefs{\CurrentBib}

\bibitem [\protect \citeauthoryear {%
Makinson%
}{%
Makinson%
}{%
{\protect \APACyear {1973}}%
}]{%
Makinson1973}
\APACinsertmetastar {%
Makinson1973}%
\begin{APACrefauthors}%
Makinson, D\BPBI C.%
\end{APACrefauthors}%
\unskip\
\newblock
\APACrefYear{1973}.
\newblock
\APACrefbtitle {Topics in Modern Logic} {Topics in modern logic}.
\newblock
\APACaddressPublisher{London}{Methuen}.
\PrintBackRefs{\CurrentBib}

\bibitem [\protect \citeauthoryear {%
Peters%
}{%
Peters%
}{%
{\protect \APACyear {1979}}%
}]{%
Peters1979}
\APACinsertmetastar {%
Peters1979}%
\begin{APACrefauthors}%
Peters, S.%
\end{APACrefauthors}%
\unskip\
\newblock
\APACrefYearMonthDay{1979}{}{}.
\newblock
{\BBOQ}\APACrefatitle {A truth-conditional formulation of Karttunen' saccount of presupposition} {A truth-conditional formulation of karttunen' saccount of presupposition}.{\BBCQ}
\newblock
\APACjournalVolNumPages{Synthese}{40}{2}{301--316}.
\PrintBackRefs{\CurrentBib}

\bibitem [\protect \citeauthoryear {%
Priest%
}{%
Priest%
}{%
{\protect \APACyear {1979}}%
}]{%
priest1979}
\APACinsertmetastar {%
priest1979}%
\begin{APACrefauthors}%
Priest, G.%
\end{APACrefauthors}%
\unskip\
\newblock
\APACrefYearMonthDay{1979}{}{}.
\newblock
{\BBOQ}\APACrefatitle {The logic of paradox} {The logic of paradox}.{\BBCQ}
\newblock
\APACjournalVolNumPages{Journal of Philosophical Logic}{8}{1}{219-241}.
\PrintBackRefs{\CurrentBib}

\bibitem [\protect \citeauthoryear {%
P\v{r}enosil%
}{%
P\v{r}enosil%
}{%
{\protect \APACyear {2023}}%
}]{%
prenosil_2023}
\APACinsertmetastar {%
prenosil_2023}%
\begin{APACrefauthors}%
P\v{r}enosil, A.%
\end{APACrefauthors}%
\unskip\
\newblock
\APACrefYearMonthDay{2023}{}{}.
\newblock
{\BBOQ}\APACrefatitle {The lattice of super-{B}elnap logics} {The lattice of super-{B}elnap logics}.{\BBCQ}
\newblock
\APACjournalVolNumPages{The Review of Symbolic Logic}{16}{1}{114–163}.
\PrintBackRefs{\CurrentBib}

\bibitem [\protect \citeauthoryear {%
Ripley%
}{%
Ripley%
}{%
{\protect \APACyear {2012}}%
}]{%
ripley2012conservatively}
\APACinsertmetastar {%
ripley2012conservatively}%
\begin{APACrefauthors}%
Ripley, E.%
\end{APACrefauthors}%
\unskip\
\newblock
\APACrefYearMonthDay{2012}{}{}.
\newblock
{\BBOQ}\APACrefatitle {Conservatively extending classical logic with transparent truth} {Conservatively extending classical logic with transparent truth}.{\BBCQ}
\newblock
\APACjournalVolNumPages{The Review of Symbolic Logic}{5}{2}{354--378}.
\PrintBackRefs{\CurrentBib}

\bibitem [\protect \citeauthoryear {%
Rivieccio%
}{%
Rivieccio%
}{%
{\protect \APACyear {2012}}%
}]{%
Rivieccio2012}
\APACinsertmetastar {%
Rivieccio2012}%
\begin{APACrefauthors}%
Rivieccio, U.%
\end{APACrefauthors}%
\unskip\
\newblock
\APACrefYearMonthDay{2012}{}{}.
\newblock
{\BBOQ}\APACrefatitle {An infinity of super-Belnap logics} {An infinity of super-belnap logics}.{\BBCQ}
\newblock
\APACjournalVolNumPages{Journal of Applied Non-Classical Logics}{22}{4}{319--335}.
\PrintBackRefs{\CurrentBib}

\bibitem [\protect \citeauthoryear {%
Szmuc%
\ \BBA {} Ferguson%
}{%
Szmuc%
\ \BBA {} Ferguson%
}{%
{\protect \APACyear {2021}}%
}]{%
szmucferguson2021}
\APACinsertmetastar {%
szmucferguson2021}%
\begin{APACrefauthors}%
Szmuc, D.%
\BCBT {}\ \BBA {} Ferguson, T\BPBI M.%
\end{APACrefauthors}%
\unskip\
\newblock
\APACrefYearMonthDay{2021}{}{}.
\newblock
{\BBOQ}\APACrefatitle {Meaningless Divisions} {Meaningless divisions}.{\BBCQ}
\newblock
\APACjournalVolNumPages{Notre Dame Journal of Formal Logic}{62}{3}{399--424}.
\PrintBackRefs{\CurrentBib}

\bibitem [\protect \citeauthoryear {%
Wintein%
}{%
Wintein%
}{%
{\protect \APACyear {2016}}%
}]{%
wintein2016all}
\APACinsertmetastar {%
wintein2016all}%
\begin{APACrefauthors}%
Wintein, S.%
\end{APACrefauthors}%
\unskip\
\newblock
\APACrefYearMonthDay{2016}{}{}.
\newblock
{\BBOQ}\APACrefatitle {On all strong {K}leene generalizations of classical logic} {On all strong {K}leene generalizations of classical logic}.{\BBCQ}
\newblock
\APACjournalVolNumPages{Studia Logica}{104}{}{503--545}.
\PrintBackRefs{\CurrentBib}

\bibitem [\protect \citeauthoryear {%
W{\'o}jcicki%
}{%
W{\'o}jcicki%
}{%
{\protect \APACyear {1988}}%
}]{%
Wojcicki1988}
\APACinsertmetastar {%
Wojcicki1988}%
\begin{APACrefauthors}%
W{\'o}jcicki, R.%
\end{APACrefauthors}%
\unskip\
\newblock
\APACrefYear{1988}.
\newblock
\APACrefbtitle {Theory of Logical Calculi: Basic Theory of Consequence Operations} {Theory of logical calculi: Basic theory of consequence operations}\ (\BVOL~199).
\newblock
\APACaddressPublisher{Dordrecht}{Springer}.
\PrintBackRefs{\CurrentBib}

\end{thebibliography}

\end{document}